\documentclass[journal]{IEEEtran}
\usepackage[dvipdfmx,hiresbb]{graphicx}
\usepackage{times}
\usepackage{bm}
\usepackage{amsfonts}
\usepackage{here}
\usepackage[dvipdfmx]{color}
\usepackage{subcaption}
\usepackage[plain,noend]{algorithm2e}
\usepackage{amsthm,amsmath,amssymb}
\newcommand{\bvec}[1]{\mbox{\boldmath $#1$}}
\newtheorem{prop}{Proposition}[section]
\newtheorem{thm}{Theorem}[section]
\newcommand{\mGamma}[2]{
\raise1.5pt\hbox{$\displaystyle \mathop{\Gamma}^{\rm m}$}\hspace{-1.5pt}_{#1}^{\, #2}}
\newcommand{\mH}[2]{
\raise1.5pt\hbox{$\displaystyle \mathop{H}^{\rm m}$}\hspace{-1.5pt}_{#1}^{\, #2}}
\newcommand{\mnabla}[2]{
\raise1.5pt\hbox{$\displaystyle \mathop{\nabla}^{\rm m}$}
\hspace{-1.5pt}_{#1}^{\, #2}}
\newcommand{\eGamma}[2]{
\raise1.5pt\hbox{$\displaystyle \mathop{\Gamma}^{\rm e}$}
\hspace{-1.5pt}_{#1}^{\, #2}}
\newcommand{\enabla}[2]{
\raise1.5pt\hbox{$\displaystyle \mathop{\nabla}^{\rm e}$}\hspace{-1.5pt}_{#1}^{\, #2}}

\ifCLASSINFOpdf
\else
\fi
\begin{document}
%
\title{Bayes Extended Estimators \\ for Curved Exponential Families}
%
%
%

\author{Michiko~Okudo,
		and~Fumiyasu~Komaki,~\IEEEmembership{Member,~IEEE}
\thanks{Manuscript received January xx, 20xx; revised January xx, 20xx; accepted
January xx, 20xx. Date of publication January xx, 20xx; date of current version January xx, 20xx.
This work was supported in part by JSPS KAKENHI Grant Number JP18J10499,
MEXT KAKENHI Grant number 16H06533, JST CREST Grant Number JPMJCR1763,
and AMED Grant Number JP19dm0207001.
}
\thanks{
M. Okudo is with the Department of Mathematical Informatics, The University of Tokyo, Tokyo 113-8656, Japan
(e-mail: {okudo@mist.i.u-tokyo.ac.jp}).}
\thanks{F. Komaki is with the Department of Mathematical Informatics, The University of Tokyo, Tokyo 113-8656, Japan,
and also with the RIKEN Center for Brain Science, Wako 351-0198, Japan (e-mail: komaki@mist.i.u-tokyo.ac.jp).}
}

%
%

\markboth{Journal of \LaTeX\ Class Files,~Vol.~xx, No.~x, August~20xx}%
{Shell \MakeLowercase{\textit{et al.}}: Bare Demo of IEEEtran.cls for IEEE Journals}
%



\maketitle

\begin{abstract}
The Bayesian predictive density has complex representation and does not belong to any finite-dimensional statistical model except for in limited situations.
In this paper, we introduce its simple approximate representation employing its projection onto a finite-dimensional exponential family.
Its theoretical properties are established parallelly to those of the Bayesian predictive density when the model belongs to curved exponential families.
It is also demonstrated that the projection asymptotically coincides with the plugin density with the posterior mean of the expectation parameter of the exponential family, which we refer to as the Bayes extended estimator.
Information-geometric correspondence indicates that the Bayesian predictive density can be represented as the posterior mean of the infinite-dimensional exponential family.
The Kullback--Leibler risk performance of the approximation is demonstrated by numerical simulations and it indicates that the posterior mean of the expectation parameter approaches the Bayesian predictive density as the dimension of the exponential family increases.
It also suggests that approximation by projection onto an exponential family of reasonable size is practically advantageous with respect to risk performance and computational cost.
\end{abstract}


%
\IEEEpeerreviewmaketitle

\section{Introduction}
Constructing predictive densities is a fundamental problem in statistical analysis that aims at predicting the behavior of future samples using past observations.
Let us suppose that we have observations
$x^n = \{x(1),x(2),\dots,x(n)\}$
that are independently distributed according to a probability distribution with density function $p(x;\omega)$ that belongs to a statistical model
\begin{align*}
\mathcal{P}
=\left\{p(x;\omega) \mid \omega \in \Omega \right\}.
\end{align*}
The objective is to provide the predictive density of $y = x(n+1)$ that is independently distributed according to the same density $p(y;\omega)$.
We adopt the Kullback--Leibler divergence
\[
D\{p(y;\omega);\hat{p}(y;{\omega}) \} = \int p(y;\omega)\log \frac{p(y;\omega)}{\hat{p}(y;{\omega})}dy
\]
as a loss function of a predictive density $\hat{p}(y;{\omega})$.
Then, the risk function and the Bayes risk with respect to a prior $\pi(\omega)$ can be written as
\begin{align*}
 &{E}[D\{p(y;\omega);\hat{p}(y;{\omega}) \}]\\
 &=\int p(x^n;\omega)D\{p(y;\omega);\hat{p}(y;{\omega}) \}dx^n,
\end{align*}
and
\[
\int\pi(\omega) \int p(x^n;\omega)D\{p(y;\omega);\hat{p}(y;{\omega}) \}dx^n d\omega,
\]
respectively.

Except for in limited cases, the Bayesian predictive distribution
does not belong to any finite-dimensional model,
which makes it intractable to obtain the full density  although it is optimal with respect to the Bayes risk.
The Bayesian predictive density is defined by
\[
 \hat{p}_\pi(y \mid x^n) = \int p(y;\omega)p_\pi(\omega \mid x^n)d\omega
\]
where $p_\pi(\omega \mid x^n)$ is the posterior density
\[
 p_\pi(\omega \mid x^n) = \frac{p(x^n;\omega)\pi(\omega)}{\int p(x^n;\omega)\pi(\omega)d\omega}
\]
of $\omega$.
It is shown in \cite{aitchison1975} that the Bayesian predictive density is optimal with respect to the Bayes risk in terms of the Kullback--Leibler divergence in the family of all probability densities, which we denote as $\mathcal{F}$.
However, the full Bayesian predictive density is intractable in most problems due to the complex representation that involves averaging plugin densities about the model parameters.
It is not included in the model $\mathcal{P}$ or even  in any finite-dimensional model in most problems, while plugin densities are always included in $\mathcal{P}$ as they are constructed by plugging-in an estimator $\hat{\omega}(x^n)$ to the model.

In the present paper, we represent the Bayesian predictive density as the infinite-dimensional limit of a parameterized distribution of an exponential family.
We demonstrate that the Bayesian predictive density can be considered as an infinite-dimensional extension of the plugin density with the posterior mean of the expectation parameter of an exponential family.
It is shown that theoretical properties including optimality with respect to the Bayes risk are appropriately retained by this extension.
The plugin density with the posterior mean of the expectation parameter coincides with the projection of the Bayesian predictive density onto the exponential family with respect to the Fisher metric.
It is shown that it approaches the Bayesian predictive density closer with respect to the risk as the projected exponential family increases.
There is also information-geometrical correspondence between the Bayesian predictive density and the projected density that comes from the correspondence between $\mathcal{F}$ and the exponential family.
In practice, the Bayesian predictive density can be  computationally approximated, for example, by taking the mean of plugin densities using Markov chain Monte Carlo simulations, or by performing an approximation of the posterior density using methods like the Laplace method, although the objective of this research is not to develop a complex approximation based on computational methods.
Apart from computational approximations, a class of empirical Bayes predictive densities is proposed for multivariate normal models in \cite{xu2011} to avoid the intractable implementations of Bayesian predictive densities.
Rather than constructing an approximation of the Bayesian predictive density that has good performance in terms of risk, we aim to formulate a simple interpretation of the Bayesian predictive density that maintains its theoretical properties on a finite-dimensional model.

The outline of the construction of the approximate predictive densities is explained below.
We consider a model of a subspace of an exponential family as $\mathcal{P}$, namely we consider a statistical model of a curved exponential family
\begin{align*}
\mathcal{P}
= &\{p(x;\omega)=s(x)
\exp(\theta^i(\omega) x_i-\Psi(\theta(\omega)))  \mid \\
&\omega=(\omega^a)\in \Omega, a=1,\dots,d, \ i=1,\dots,m \},
\end{align*}
where $\Omega \subset \mathbb{R}^d$ and $1 \leq d \leq m$.
The model $\mathcal{P}$ parametrized by $\omega$ is embedded in an exponential family parametrized by $\theta$, thus here we represent $\theta$ as $\theta(\omega)$.
Summation over a repeated index is automatically taken according to Einstein's summation convention:
if an index occurs as an upper and lower index in one term, then the summation is implied.
Curved exponential families embedded in exponential families can express a variety of models including network models (e.g., \cite{hunter2006}) and time series models (e.g., \cite{ravishanker1990}).
They can also be applied to stochastic processes \cite{kuchler2006}.
We consider predictive densities in a finite-dimensional full exponential family
\begin{align*}
\mathcal{E}
= &\{ p(x;\theta)=s(x)\exp(\theta^i x_i-\Psi(\theta)) \mid \theta=(\theta^i)\in \Theta, \\
& i=1,\dots,m \}~~(\Theta\subseteq \mathbb{R}^m)
\end{align*}
that includes the original curved exponential model $\mathcal{P}$.
We refer to plugin densities in $\mathcal{E}$ as extended plugin densities.
The inclusion relation is
$\mathcal{P} \subseteq \mathcal{E} \subseteq \mathcal{F}$
and we consider the middle layer of the three-layer structure.
The coordinate system $\theta=(\theta^i) \ (i=1,\dots,m)$ is called the natural parameter of exponential families.
Another coordinate system $\eta = (\eta_i)$ defined by
\begin{align}
\label{eq:eta}
 \eta_i  = {E}(x_i) = \frac{\partial}{\partial \theta^i}\Psi(\theta) ~~~~~ (i=1,\dots,m)
\end{align}
is called the expectation parameter.
The posterior mean of $\eta$ is closely related to the Bayesian predictive density,
and the extended plugin density with the posterior mean of $\eta$ is considered in this paper.
Based on the idea of covering $\mathcal{F}$ by extending exponential families, we specify the models that can be embedded in exponential families, and the theoretical properties described in the following sections are based on this embedding.
The policy of expressing a probability density by extending  exponential families has been investigated, for example, in \cite{barron1991}, in which a log-density is approximated using series of polynomials and the rate of convergences is obtained.
It should be noted that the practical advantage illustrated in the numerical experiments in Section \ref{sec:numerical} can be attained for other models if it is possible to find an appropriate exponential family onto  which we project the Bayesian predictive density.

\begin{table}[h]
\begin{center}
\caption{{All probability densities and a full exponential family}}
  \begin{tabular}{l|l} \hline
 {$\mathcal{F}$: all probability densities} & {$\mathcal{E}$: a full exponential family} \\  \hline
{$p(x)$: m-representation} & {$\eta$: m-affine parameter} \\
{$\log p(x)$: e-representation} & {$\theta$: e-affine parameter} \\
{infinite dimensional model} & {finite dimensional model} \\ \hline 
  \end{tabular}
  \label{tab:infinite}
\end{center}
\end{table}
From the viewpoint of information geometry, the posterior mean of $\eta$ can be considered as the correspondence to the Bayesian predictive density in $\mathcal{E}$.
Table \ref{tab:infinite} represents the infinite-finite correspondence of m and e representations.
Here, ``m" and ``e" are short notations denoting ``mixture" and ``exponential," respectively.
Exponential families and mixture families are important dual families in information geometry, and their typical representations are denoted as m-representation and e-representation, respectively.
Concerning exponential families, 
the e-representation is $\log p(x) = \theta^i x_i-\Psi(\theta) + \log s(x),$ and 
$\theta$ is called the e-affine parameter, as the basis vector fields $(\partial/\partial \theta_ i)p(x)~(i=1,\dots,m)$ are parallel vector fields with respect to the e-connection as defined in Section \ref{sec:preliminaries} of this paper.
Concerning mixture families, the m-representation can be written as  $p(x)=\eta_i q^i(x) + c(x)$, and $\eta$ is the affine coordinate system about the m-connection (also defined in Section \ref{sec:preliminaries}).
We can also set the m-affine coordinate system in exponential families, and it is defined by (\ref{eq:eta}).
Here, the m-affine (or e-affine) parameters are the finite-dimensional typical representations of m-representation (or e-representation, respectively).
The Bayesian predictive density is the posterior mean about the m-representation (that is, density functions), and its finite-dimensional correspondence is the posterior mean about $(\eta_i)$.
Since the Bayesian predictive density is optimal in the infinite-dimensional exponential family, we might expect that the posterior mean of $(\eta_i)$ exhibit the same properties as the Bayesian predictive density, such as  optimality with respect to the Bayes risk.

The properties of the posterior mean of ${\eta}$ are investigated in following sections as follows.
In Section \ref{section_bayes_estimator}, we show that the extended plugin density with the posterior mean of $\eta$ is optimal with respect to the Bayes risk in the finite-dimensional exponential family.
We denote the posterior mean of $\eta$ as the Bayes extended estimator.
In Section \ref{subsec:projection}, the extended plugin density with the Bayes extended estimator is proved to be the projection of the Bayesian predictive density onto $\mathcal{E}$ in terms of the Fisher metric.
In Section \ref{sec:optimal_shift}, its optimality with respect to the risk along orthogonal shift from the model is shown to be common with the property of optimality in the case of the Bayesian predictive density.
The relation between the projection angle under the Fisher metric and the risk difference between the Bayesian predictive density and the extended plugin density with the Bayes extended estimator is investigated.
In Section \ref{sec:numerical}, we compare the risk performance of the Bayes estimator, the Bayes extended estimator, and the Bayesian predictive density by conducting numerical simulations on the Gaussian spiked covariance models.
We confirm that the projection angle converges to zero as the dimension of $\mathcal{E}$ increases.
The simulation results also suggest that the projection of the Bayesian predictive density is practically effective in approximating it with respect to the Kullback--Leibler risk and the computational cost.

\section{Preliminaries}
\label{sec:preliminaries}
In this section, we prepare some information-geometric notions.
For details of the notions and notation concerning the differential geometry of curved exponential families, refer to \cite{amari1985}.

Let $a,b,\dots$ be indices for $\omega$.
Let $T_\omega \mathcal{P}$ be the tangent space of $\mathcal{P}$ at a point $\omega$.
The tangent space $T_\omega \mathcal{P}$ is identified with the vector space spanned by $\partial_a p(x;\omega)~(a=1,\dots,d)$, where $\partial_a$ denotes $\partial/\partial \omega^a$.
We define inner products in the tangent space by
\begin{align}
\label{inner product}
{
\langle \partial_a p(x;\omega), \partial_b p(x;\omega) \rangle = \int \frac{\partial_a p(x;\omega)\partial_b p(x;\omega)}{p(x;\omega)}dx.
}
\end{align}
In a statistical model
$\mathcal{P},$ each component of the Fisher information matrix is defined by
\[
g_{ab}(\omega) = \langle \partial_a p(x;\omega), \partial_b p(x;\omega) \rangle.
\]
Let $g^{ab}$ be a component of the inverse matrix of $(g_{ab})$.
Then, e-connection and the m-connection coefficients are defined as:
\begin{align*}
\eGamma{abc}{}(\omega)
&=\int p(x;\omega)\{\partial_a \partial_b \log p(x;\omega)\}\{\partial_c \log p(x;\omega)\}dx
\end{align*}
and
\begin{align*}
\mGamma{abc}{}(\omega)
&=\int \frac{\partial_a \partial_b p(x;\omega) \partial_c p(x;\omega)}
{p(x;\omega)}dx,
\end{align*}
respectively.
We define
\[
 \eGamma{ab}{~c}=\eGamma{abd}{} g^{dc},~~\mGamma{ab}{~c}=\mGamma{abd}{} g^{dc},
\]
and
\begin{align*}
T_{abc}
=&\mGamma{abc}{}-\eGamma{abc}{}\\
=&\int p(x;\omega)\{\partial_a\log p(x;\omega)\}
\{\partial_b\log p(x;\omega)\}\\
&\times\{\partial_c\log p(x;\omega)\}
dx.
\end{align*}
The Jeffreys prior density is given by
\[
 \pi_J(\omega)=\sqrt{|g(\omega)|},
\]
where $|g(\omega)|$ is the determinant of the matrix $(g_{ab}(\omega))$.

The coordinate systems $(\theta^i)$ and $(\eta_i)$ of the exponential family $\mathcal{E}$ are
dual to each other in the sense that
\begin{align}
\label{dual}
\left\langle \frac{\partial}{\partial \theta^i}p(x;\theta),  \frac{\partial}{\partial \eta_j}p(x;\theta) \right\rangle = \delta^i_j
\end{align}
where $\delta^i_j$ is the Kronecker delta.
In a curved exponential family, e-connection and m-connection coefficients
are expressed as:
\begin{align}
\label{mGamma simple}
\eGamma{abc}{} = (\partial_{a}\partial_{b}\theta^i) (\partial_c \eta_i) \mbox{~~and~~}
\mGamma{abc}{} = (\partial_{a}\partial_{b}\eta_i) (\partial_c \theta^i),
\end{align}
respectively.

{
In the rest of the paper, we assume regularity conditions to ensure that equalities such as
\[
\int \rm{o}_p(1)dy = \rm{o}(1), ~\int \rm{O}_p(n^{-1})dy = \rm{O}(n^{-1}),
\]
hold.
For the details of the regularity conditions, see \cite{hartigan1998}.
}

\section{Main results}
\subsection{Optimality with respect to Bayes risk}
\label{section_bayes_estimator}
The posterior mean of $\eta$ (we denote as $\bar{\eta}_\pi$) of $\mathcal{E}$ is evaluated as follows:
\[
\bar{\eta}_\pi=\frac{\int\eta(\omega) p(x^n;\eta(\omega))\pi(\omega)d\omega}{\int p(x^n;\eta(\omega))\pi(\omega)d\omega}.
\]
Note that $\bar{\eta}_\pi(x^n) \neq \eta(\bar{\omega}_\pi(x^n))$ in general where $\bar{\omega}_\pi(x^n)$ is the posterior mean of $\omega$:
\[
\bar{\omega}_\pi=\frac{\int\omega p(x^n;\omega)\pi(\omega)d\omega}{\int p(x^n;\omega)\pi(\omega)d\omega}.
\]
We demonstrate that $p(y;\bar{\eta}_\pi)$ is optimal in $\mathcal{E}$
with respect to the Bayes risk based on a prior $\pi$.

\begin{prop}
\label{thm_E-Bayes}
The Bayes risk of $p(y;\hat{\eta})$, where $\hat{\eta}$ is an estimator of $\eta$, is minimized when $\hat{\eta}=\bar{\eta}_\pi.$
\end{prop}

\begin{proof}
Let $\hat{\theta}$ be an estimator of $\theta$.
Note that $\theta$ and $\eta$ are functions of $\omega$.
The Kullback--Leibler loss of ${p}(y;\hat{\theta})$ is
\begin{align*}
&D\{p(y;\theta(\omega));{p}(y;\hat{\theta}) \}\\
&=\int p(y;\theta)\log\left(\frac{\exp(\theta^i y_i -\Psi(\theta))}
{\exp(\hat{\theta}^i y_i -\Psi(\hat{\theta}))}\right)dy\\
&= (\theta^i-\hat{\theta}^i)\eta_i-(\Psi(\theta)-\Psi(\hat{\theta})).
\end{align*}
Hence
\begin{align}
\label{eq:prop_1}
&\int p_\pi(\omega \mid x^n) D\{p(y;\theta(\eta));{p}(y;\hat{\theta}(\eta))\}d\omega \nonumber \\
&=\ (\overline{\theta^i\eta_i}-\hat{\theta}^i\overline{\eta_i}) - (\overline{\Psi(\theta)}-\Psi(\hat{\theta})) \nonumber \\
&=\ ({\theta^i(\overline{\eta})}-\hat{\theta}^i)\overline{\eta}^i - ({\Psi(\theta(\overline{\eta}))}-\Psi(\hat{\theta})) \nonumber \\
& \ \ \ \ + \left( -\theta^i(\overline{\eta})\overline{\eta_i} + \overline{\theta^i\eta_i} + \Psi(\theta(\overline{\eta})) - \overline{\Psi(\theta)}\right) \nonumber \\
&=\ D\{p(y;\overline{\eta});{p}(y;\hat{\theta}) \} \nonumber\\
& \ \ \ \ + \left( -\theta^i(\overline{\eta})\overline{\eta}^i + \overline{\theta^i\eta_i} + \Psi(\theta(\overline{\eta})) - \overline{\Psi(\theta)}\right),
\end{align}
where, for a function $f(\eta)$,
\[
\overline{f(\eta)}= \int p_\pi(\omega \mid x^n) f(\eta)d\omega.
\]
It is minimized when $\hat{\theta}=\theta(\overline{\eta})=\theta(\overline{\eta}_\pi)$.
{
By multiplying (\ref{eq:prop_1}) with
$\int p(x^n;\omega)\pi(\omega)d\omega$
and then integrating with respect to $x^n$,}
 it is shown that $p(y; \bar{\eta}_\pi)$ is optimal with respect to the Bayes risk in $\mathcal{E}$.
\end{proof}

We refer to $\bar{\eta}_\pi$ as the Bayes extended estimator.
Hereinafter, we denote the Bayes extended estimator and the Bayes estimator of $\omega$ about a prior $\pi$ as $\hat{\eta}_\pi(=\bar{\eta}_\pi)$ and $\hat{\omega}_\pi$, respectively. 

From Proposition \ref{thm_E-Bayes}, the extended plugin density with $\hat{\eta}_\pi$ is 
the projection of the Bayesian predictive density onto $\mathcal{E}$ about the Bayes risk.
It is nearest to the Bayesian predictive density in $\mathcal{E}$ regarding the Bayes risk, because the Bayesian predictive density is optimal about the Bayes risk in $\mathcal{F}$.
In fact, $p(y,\hat{\eta}_\pi)$ coincides with the projection of the Bayesian predictive density onto $\mathcal{E}$ regarding the Fisher metric asymptotically, as shown in Section \ref{subsec:projection}.

The choice of $\mathcal{E}$ does not require to be fixed, and we can consider situations in which the size of the extended model $\mathcal{E}$ can be increased, for example, by employing sequences of exponential families as in \cite{barron1991} and \cite{barron1992}.
In those situations, the extended plugin density with the Bayes extended estimator $\hat{\eta}_\pi$ approaches the Bayesian predictive density as $\mathcal{E}$ grows, as $\mathcal{E}$ approaches the set of all probability distributions $\mathcal{F}$.

Here, we use a simple example to illustrate the difference of the plugin density with $\hat{\omega}_\pi$, the extended plugin density with $\hat{\eta}_\pi$, and the Bayesian predictive density. \vspace{2mm}

\noindent {Example ~(Fisher circle model)}
We consider two dimensional Gaussian distribution $\mathrm{N}(\mu,I_2)$
with unknown mean vector $\mu$ and the identity covariance matrix $I_2$.
The density function is
\begin{align*}
p(x;\mu)=&\frac{1}{2\pi}\exp
\left[-\frac{1}{2}\left\{(x_1-\mu_1)^2 + (x_2-\mu_2)^2 \right\}\right]\\
=&\frac{1}{2\pi}\exp\left(-\frac{1}{2}(x_1^2 + x_2^2)\right) \\
 & \times \exp\left(x_1\mu_1 + x_2\mu_2 - \frac{1}{2}\left(\mu_1^2 + \mu_2^2 \right)\right).
\end{align*}
When the mean vector $\mu$ is expressed as
\[\mu_1 = \cos\omega, \mu_2 = \sin \omega,\]
the one-dimensional submodel is called the Fisher circle model. Here, the following holds:
\begin{align*}
\theta^1 = \eta_1 =\mu_1&, \ \theta^2 = \eta_2 =\mu_2.
\end{align*}

Then, we derive the Bayes estimator $\hat{\omega}_\pi$, the Bayes extended estimator $\hat{\eta}_\pi$, and the Bayesian predictive density.
For \(x^n = \{x(1),x(2),\dots,x(n)\}\),
\begin{align*}
&p(x^n;\omega)\\
&= \prod_{t=1}^n \frac{1}{2\pi}\exp\left(-\frac{||x(t)-\mu(\omega)||^2}{2} \right)\\
&= \frac{1}{(2\pi)^n}
\exp\left(-\frac{\sum_{t=1}^n (x_1(t)^2+x_2(t)^2)}{2}+\frac{n}{2}(\bar{x}_1^2+\bar{x}_2^2)\right)\\
& \ \ \ \ \times \exp\left(-\frac{n}{2}||\bar{x}-\mu(\omega)||^2 \right)
\end{align*}
where \(\bar{x}=\sum^n_{t=1}x(t)/n\).
Let \(\bar{x}=(\|\bar{x}\|\cos\phi, \|\bar{x}\|\sin \phi)^\top\).
{Then by the law of cosines,}
\begin{align*}
&-\frac{n}{2}\|\bar{x}-\mu(\omega)\|^2 \\
&= n||\bar{x}||\cos(\omega-\phi)
+ {\rm (terms~independent~of~}\omega)
\end{align*}
and
\[
\int_0^{2\pi} \exp(n \|\bar{x}\| \cos(\omega-\phi))d\omega = 2\pi I_0(n\|\bar{x}\|)
\]
where \(I_0(\cdot)\) is the modified Bessel function of the first kind.
See \cite{fisher1973} (pp.\ 138--140) for the {details}.
When the uniform prior
\[
\pi(\omega)\propto 1
\]
is adopted, the posterior density is
\[
p_\pi(\omega \mid x^n) = \frac{1}{2\pi I_0(n\|\bar{x}\|)}\exp(n\|\bar{x}\|\cos(\omega-\phi)).
\]
It follows that the plugin of the Bayes estimator is $\mathrm{N}\left(\frac{\bar{x}}{\|\bar{x}\|},I_2\right)$:
\[
p(y;\hat{\omega}_\pi) = \frac{1}{2\pi}\exp\left(-\frac{1}{2}\left\|y- \frac{\bar{x}}{\|\bar{x}\|}\right\|^2\right).
\]
The extended plugin with the Bayes extended estimator is $\mathrm{N}\left(\frac{I_1(n \|\bar{x}\|)}{I_0(n \|\bar{x}\|)} \frac{\bar{x}}{\|\bar{x}\|},I_2\right)$:
\begin{align*}
p(y;\hat{\eta}_\pi) = \frac{1}{2\pi}\exp\left(-\frac{1}{2}\left\|y- \frac{I_1(n \|\bar{x}\|)}{I_0(n \|\bar{x}\|)} \frac{\bar{x}}{\|\bar{x}\|} \right\|^2\right),
\end{align*}
where $\hat{\eta}_\pi$ is not included in the circle parametrized by $\omega$ and $\hat{\eta}_\pi \neq \eta(\hat{\omega}_\pi)$.
Here \(I_1(\cdot)\) is the modified Bessel function of the first kind.
On the other hand, the Bayesian predictive density is given by
\[
p_\pi(y\mid x^n) = \frac{1}{2 \pi}\frac{I_0(\|y+n \bar{x}\|)}{I_0(n \|\bar{x}\|)}
 \exp \biggl\{ - \frac{1}{2} (\|y\|^2+1) \biggr\}.
\]
Therefore, $p_\pi(y\mid x^n)$ is not included in $\mathcal{P}$ or $\mathcal{E}$
because it is not a two-dimensional Gaussian with a covariance matrix $I_2$.

\subsection{Projection of Bayesian predictive densities in terms of Fisher metric}
\label{subsec:projection}
Here, we demonstrate that $p(y;\hat{\eta}_\pi)$ is the projection of the Bayesian predictive density regarding the Fisher metric.
It is shown via asymptotic expansion of $p(y;\hat{\eta}_\pi)$, which is represented as a point in $\mathcal{E}$ that is parallelly and orthogonally shifted from the plugin density of the maximum likelihood estimator of $\mathcal{P}$ as shown in Figure \ref{projection}.
\begin{figure}[H]
\centering
\includegraphics[width=0.9\columnwidth]{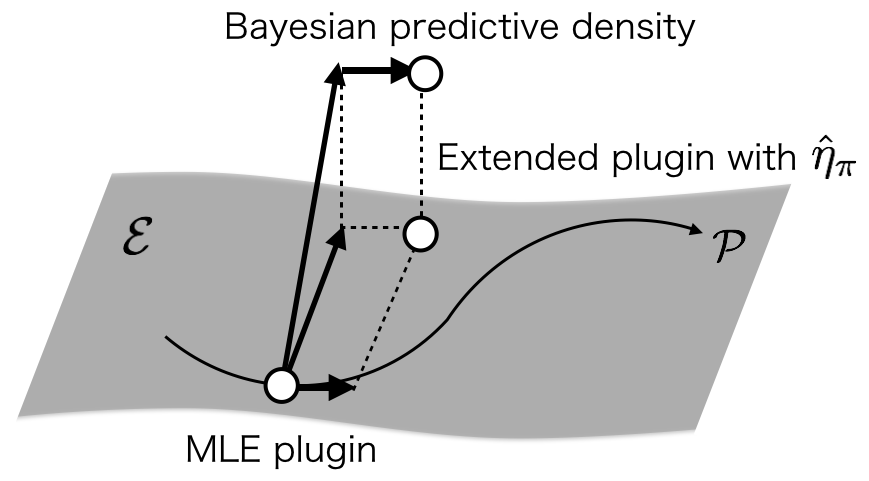}
\caption{
The extended plugin density with $\hat{\eta}_\pi$ is the projection of the Bayesian predictive density onto $\mathcal{E}$.
It is constructed by shifting from the plugin density with the maximum likelihood estimator (MLE).
}
\label{projection}
\end{figure}

Here, we proceed to obtain the asymptotic expansion of $\hat{\eta}_\pi$ around $\eta(\hat{\omega}_{\rm MLE})$.
\begin{thm}
\label{thm_estimator}
The Bayes extended estimator based on a prior $\pi(\omega)$ is expanded as
\begin{align*}
&\hat{\eta}_\pi\\
&=
\eta(\hat{\omega}_{\rm MLE})\\
&~~+\frac{{g}^{ab}(\hat{\omega}_{\rm MLE})}{2n}
\left(\partial_{a}\partial_{b}\eta(\hat{\omega}_{\rm MLE})-\mGamma{ab}{~c}(\hat{\omega}_{\rm MLE})\partial_c {\eta}(\hat{\omega}_{\rm MLE})\right)\\
&~~+\frac{{g}^{ab}(\hat{\omega}_{\rm MLE})}{n}
\left(\partial_b\log\frac{\pi}{{\pi}_J}(\hat{\omega}_{\rm MLE})+\frac{T_b(\hat{\omega}_{\rm MLE})}{2}\right)\\
&~~\times \partial_a{\eta}(\hat{\omega}_{\rm MLE})\\
&~~+\mathrm{o}_p(n^{-1}),
\end{align*}
where
{${\pi}_J$ is the density of the Jeffreys prior and}
$T_a=T_{abc}g^{bc}$.
\end{thm}

\begin{proof}
See Appendix A.
\end{proof}

We can obtain the asymptotic expansion of the extended plugin density with $\hat{\eta}_\pi$.

\begin{thm}
\label{plug-in expansion}
The extended plugin density with $\hat{\eta}_\pi$ is expanded as
\begin{align*}
&p(y;\hat{\eta}_\pi)\\
&=\ p(y;\eta(\hat{\omega}_{\rm MLE}))\\
&~~+\frac{{g}^{ab}(\hat{\omega}_{\rm MLE})}{2n}
\left(\partial_{a}\partial_{b}\eta_i(\hat{\omega}_{\rm MLE})-\mGamma{ab}{~c}(\hat{\omega}_{\rm MLE})\partial_c {\eta}_i(\hat{\omega}_{\rm MLE})\right)\\
&~~\times \partial^i p(y;\hat{\omega}_{\rm MLE})\\
&~~+\frac{{g}^{ab}(\hat{\omega}_{\rm MLE})}{n}
\left(\partial_b\log\frac{\pi}{\pi_J}(\hat{\omega}_{\rm MLE})+\frac{T_b(\hat{\omega}_{\rm MLE})}{2}\right)\\
&~~\times\partial_a p(y;\hat{\omega}_{\rm MLE})\\
&~~+\mathrm{o}_p(n^{-1}),
\end{align*}
where $\partial^i = {\partial}/{\partial\eta_i}$.
\end{thm}

\begin{proof}
Symbols such as
$\eta(\hat{\omega}_{\rm MLE}), \partial_a \eta(\hat{\omega}_{\rm MLE})$, and $\partial_a\partial_b \eta(\hat{\omega}_{\rm MLE})$
are abbreviated to
$\hat{\eta}, \partial_a\hat{\eta}$, and $\partial_{ab}\hat{\eta}$, 
respectively.
Considering the asymptotic expansion introduced in Theorem \ref{thm_estimator}, we obtain the following:
\begin{align*}
&p(y;\hat{\eta}_\pi)\\
&= p(y;\hat{\eta})
+\left(\frac{\hat{g}^{ab}}{2n}
\left( \partial_{ab}\hat{\eta}_i - \mGamma{ab}{~c}(\hat{\omega}_{\rm MLE}) \partial_c\hat{\eta}_i \right) \right. \\
&\left. ~+\frac{\hat{g}^{ab}}{n}
\left( \partial_b \log \frac{\hat{\pi}}{\hat{\pi_J}} + \frac{T_{b}(\hat{\omega}_{\rm MLE})}{2}
\right) \partial_a\hat{\eta}_i \right)
\partial^i p(y;\hat{\eta})\\
& ~~+ \mathrm{o}_p(n^{-1})\\
&= p(y;\hat{\eta}) +
\frac{\hat{g}^{ab}}{2n}
\left( \partial_{ab}\hat{\eta}_i - \mGamma{ab}{~c}(\hat{\omega}_{\rm MLE}) \partial_c \hat{\eta}_i \right)\partial^i p(y;\hat{\eta})\\
& ~~+\frac{\hat{g}^{ab}}{n}
\left( \partial_b \log \frac{\hat{\pi}}{\hat{\pi}_J} + \frac{T_{b}(\hat{\omega}_{\rm MLE})}{2}
\right)
\partial_a p(y;\hat{\eta}) + \mathrm{o}_p(n^{-1}).
\end{align*}
\end{proof}

The shift from $p(y;\eta(\hat{\omega}_{\rm MLE}))$ to $p(y;\hat{\eta}_\pi)$ in Theorem \ref{plug-in expansion} is composed of two components,  one ``parallel'' and the other ``orthogonal'' to the model $\mathcal{P}$.
That is, the term
\[
\frac{{g}^{ab}(\hat{\omega}_{\rm MLE})}{n}
\left(\partial_b\log\frac{\pi}{\pi_J}(\hat{\omega}_{\rm MLE})+\frac{T_b(\hat{\omega}_{\rm MLE})}{2}\right)\partial_a p(y;\hat{\omega}_{\rm MLE})
\]
is included in the tangent space spanned by $\partial_a p(y;\eta)~(a=1,\dots,d)$ and the term
\begin{align}
\label{orthogonal shift}
&\frac{{g}^{ab}(\hat{\omega}_{\rm MLE})}{2n}
\left(\partial_{a}\partial_{b}\eta_i(\hat{\omega}_{\rm MLE})-\mGamma{ab}{~c}(\hat{\omega}_{\rm MLE})\partial_c {\eta}_i(\hat{\omega}_{\rm MLE})\right) \nonumber \\
&\times \partial^i p(y;\hat{\omega}_{\rm MLE})
\end{align}
is orthogonal to $\partial_a p(x;\eta)~(a=1,\dots,d)$ with respect to the inner product (\ref{inner product}),
because
\begin{align*}
&\biggl\langle \bigl( \partial_{a}\partial_{b}\eta_i - \mGamma{ab}{~c} \partial_c \eta_i \bigr) \partial^i p(y;\eta), \partial_e p(y;\eta) \biggr\rangle \\
&=
\int
\partial_{a}\partial_{b}\eta_i
\frac{\partial p(y;\eta)}{\partial\eta_i}
\frac{\partial \theta^j}{\partial \omega^e}\frac{\partial p(y;\eta)}{\partial\theta^j}
\frac{1}{p(y;\eta)}dy
-
\mGamma{ab}{~c}g_{ce}\\
&=
\partial_{a}\partial_{b}\eta_i\partial_e \theta^i
-
\mGamma{abe}{}\\
&=0
\end{align*}
by using \eqref{dual} and \eqref{mGamma simple}.

{
We can compare the orthogonal shifts (\ref{orthogonal shift}) to the orthogonal shifts from $p(y;\eta(\hat{\omega}_{\rm MLE}))$ to the Bayesian predictive density, and we show (\ref{orthogonal shift}) is the projection of the orthogonal shifts to the Bayesian predictive density onto $\mathcal{E}$.}
In \cite{komaki1996}, the Bayesian predictive density $p_\pi (y\mid x)$ is asymptotically expanded as
\begin{align*}
&p_\pi (y\mid x)\\
&=\ p(y;{\eta}(\hat{\omega}_{\rm MLE}))
+ \frac{{g}^{ab}(\hat{\omega}_{\rm MLE})}{2n} \\
& ~~\times \left(
\partial_{a}\partial_{b}p(y;\hat{\omega}_{\rm MLE})-\mGamma{ab}{~c}(\hat{\omega}_{\rm MLE})\partial_c p(y;\hat{\omega}_{\rm MLE})\right)\\
&~~+\frac{{g}^{ab}(\hat{\omega}_{\rm MLE})}{n}\\
& ~~\times \left( \partial_b \log \frac{{\pi}}{{\pi}_J}(\hat{\omega}_{\rm MLE}) + \frac{T_{b}(\hat{\omega}_{\rm MLE})}{2}
\right)\partial_a p(y;\eta(\hat{\omega}_{\rm MLE}))\\
&~~+ \mathrm{o}_p(n^{-1}).
\end{align*}
The parallel shift is identical to that of $p(y;\hat{\eta}_\pi)$.
On the other hand, the orthogonal shift
\begin{align}
\label{orthogonal_komaki}
\frac{{g}^{ab}(\hat{\omega}_{\rm MLE})}{2n}\left(
\partial_{a}\partial_{b}p(y;\hat{\omega}_{\rm MLE})-\mGamma{ab}{~c}(\hat{\omega}_{\rm MLE})\partial_c p(y;\hat{\omega}_{\rm MLE})\right)
\end{align}
is different from that of $p(y;\hat{\eta}_\pi)$ and it is not included in the tangent space of $\mathcal{E}$.
Therefore, the shifted density $p_\pi (y\mid x)$ is not included in $\mathcal{E}$, while $p(y;\hat{\eta}_\pi)$ {is in $\mathcal{E}$}.

To cope with the shifts that are orthogonal to $\mathcal{P}$, we introduce a coordinate system to the subspace of $\mathcal{E}$ that is orthogonal to $\mathcal{P}$.
{We divide} the tangent vectors of $\mathcal{E}$ at $\eta$ into two parts, namely, into those parallel to $\mathcal{P}$ and those orthogonal to $\mathcal{P}$.
For each point $\eta\in\mathcal{E}$, the tangent space $T_{\eta}\mathcal{E}$ is identified with the vector space spanned by
\[
\frac{\partial}{\partial \eta_i} p(x;\eta)~(i=1,\dots,m).
\]
The tangent space $T_{\omega}\mathcal{P}$ is a subspace of $T_{\eta}\mathcal{E}$.
Let $A(\omega)$ be an $(m-d)$-dimensional smooth submanifold of $\mathcal{E}$
attached to each point $\omega\in\mathcal{P}$ and
assume that $A(\omega)$ orthogonally transverses $\mathcal{P}$ at $\eta(\omega)$.
Such a family of submanifolds $\{A(\omega) \mid \omega \in \Omega \}$ is called an ancillary family.
We introduce an adequate coordinate system
$\xi=(\xi^\kappa)~(\kappa=d+1,\dots,m)$ to $A(\omega)$
so that a pair $(\omega,\xi)$ uniquely specifies a point of $\mathcal{E}$ in the neighborhood of $\eta(\omega)$.
We adopt a coordinate system $\xi$ on $A(\omega)$ so that $\eta(\omega,\xi) \in \mathcal{P}$ if $\xi = 0$.
Then, we have
\[
\mathrm{span} \left\{\partial^i p(x;\eta)\right\}
 = \mathrm{span} \left\{\partial_a p(x;\eta), \partial_\kappa p(x;\eta)\right\}
\]
where $\partial_\kappa p(x;\eta) = ({\partial}/{\partial \xi^\kappa}) p(x;\eta)$.
Since $A(\omega)$ orthogonally transverses $\mathcal{E}$, we have
\[
\langle \partial_a p(x;\eta), \partial_\kappa p(x;\eta)\rangle = 0~~(a=1,\dots,d,~\kappa=d+1,\dots,m).
\]

Now we show that the extended plugin density with $\hat{\eta}_\pi$ is the projection of the Bayesian predictive density onto $\mathcal{E}$ asymptotically, as shown in Figure \ref{projection}.
The projection of \eqref{orthogonal_komaki}
onto the tangent space of $\mathcal{E}$ is
\begin{align}
\label{eq:projection_orthogonal}
&\left\langle
\frac{{g}^{ab}(\hat{\omega}_{\rm MLE})}{2n}
(
\partial_{a}\partial_{b}p(y;\hat{\omega}_{\rm MLE})-\mGamma{ab}{~c}(\hat{\omega}_{\rm MLE})\partial_c p(y;\hat{\omega}_{\rm MLE})), \right. \nonumber \\
&\left. \partial_\lambda p(y;\hat{\omega}_{\rm MLE})
\right\rangle g^{\lambda\kappa}\partial_\kappa p(y;\hat{\omega}_{\rm MLE}) 
~~{(
\lambda=d+1,\dots,m
).}
\end{align}
{Because $\langle \partial_c p(x;\eta), \partial_\lambda p(x;\eta) \rangle=0~(c=1,\dots,d)$, (\ref{eq:projection_orthogonal}) is}
\begin{align*}
&\frac{{g}^{ab}(\hat{\omega}_{\rm MLE})}{2n}\left\langle
\partial_{a}\partial_{b}p(y;\hat{\omega}_{\rm MLE}), \partial_\lambda p(y;\hat{\omega}_{\rm MLE})
\right\rangle\\
& \times g^{\kappa \lambda}\partial_\kappa p(y;\hat{\omega}_{\rm MLE})\\
&= \frac{{g}^{ab}(\hat{\omega}_{\rm MLE})}{2n}\mH{ab}{\kappa}
\partial_\kappa p(y;\hat{\omega}_{\rm MLE})
\end{align*}
where
\begin{align*}
\mH{ab\kappa}{}
&{
=\mGamma{ab\kappa}{}
}\\
&=\langle \partial_a\partial_b p(x;\omega),\partial_\kappa p(x;\eta) \rangle\\
&= (\partial_a\partial_b \eta_i)(\partial_\kappa \theta^i)
\end{align*}
{$(a,b=1,\dots,d,~\kappa=d+1,\dots,m)$} is the mixture embedding curvature of $\mathcal{P}$ in $\mathcal{E}$ and
$\mH{ab}{~\kappa}=\mH{ab\lambda}{}g^{\kappa\lambda}$.
{
Here we represent $\mGamma{ab\kappa}{}$ as $\mH{ab\kappa}{}$ to clarify its geometrical interpretation as the mixture embedding curvature, distinguishing from $\mGamma{abc}{}, \mGamma{a\kappa b}{}$ and $\mGamma{\kappa ab}{}.$
}
We confirm that it coincides with the orthogonal shift to $p(y;\hat{\eta}_\pi)$ asymptotically as follows.
Let
\[
h_{ab} = (\partial_{a}\partial_{b}\eta_i({\omega}))\partial^i p(x;\eta)-\mGamma{ab}{~c}\partial_c p(x;\omega),
\]
then the orthogonal component in Theorem \ref{plug-in expansion} is
\begin{align}
\label{eq:hab}
\frac{g^{ab}(\hat{\omega}_{\rm MLE})}{2n}h_{ab}(\hat{\omega}_{\rm MLE}).
\end{align}
Since $h_{ab}~(a,b=1,\dots,d)$ are included in the space spanned by $\partial_\kappa p(x;\omega) ~(\kappa=d+1,\dots,m)$, the following holds:
\begin{align*}
h_{ab}
&=\langle h_{ab},\partial_\lambda p(x;\eta) \rangle g^{\kappa \lambda}\partial_\kappa p(x;\eta)\\
&=\left\langle
\partial_{a}\partial_{b}\eta_i({\omega})\partial^i p(x;\eta)-\mGamma{ab}{~c}\partial_c p(x;\omega),  \partial_\lambda p(x;\eta)\right\rangle\\
& \ \ \ \ \times g^{\kappa \lambda}\partial_\kappa p(x;\eta).
\end{align*}
As $\langle \partial_c p(x;\eta), \partial_\lambda p(x;\eta) \rangle=0$\\
$~{(
c=1,\dots,d,~
\lambda=d+1,\dots,m
)}
$,
\begin{align}
\label{eq:hab1}
h_{ab}
&= \langle \partial_{a}\partial_{b}\eta_i({\omega})\partial^i p(x;\eta), \partial_\lambda p(x;\omega) \rangle g^{\kappa \lambda}\partial_\kappa p(x;\eta) \nonumber \\
&= \partial_{a}\partial_{b}\eta_i({\omega})\left\langle \partial^i p(x;\omega), \frac{\partial \theta^j}{\partial \xi^\lambda}\frac{\partial p(x;\eta)}{\partial \theta^j} \right\rangle g^{\kappa \lambda}\partial_\kappa p(x;\eta) \nonumber \\
&= \partial_{a}\partial_{b}\eta_i({\omega})\partial_\lambda \theta^i
g^{\kappa \lambda}\partial_\kappa p(x;\eta) \nonumber \\
&= \mH{ab}{~\kappa}\partial_\kappa p(x;\eta).
\end{align}

\subsection{Projection angle and risk difference}
\label{sec:optimal_shift}
In this section, we demonstrate that the extended plugin density with $\hat{\eta}_\pi$ is optimal with respect to  the risk along the orthogonal shift from the model $\mathcal{P}$.
This property is parallel to those of Bayesian predictive densities investigated in \cite{komaki1996}, as explained later in this section.
Orthogonal shifts from $\mathcal{P}$ can asymptotically improve the Kullback--Leibler risk from plugin densities of $\mathcal{P}$.
By evaluating those risk improvements, the risk difference between $p(y;\hat{\eta}_\pi)$ and Bayesian predictive densities can be corresponded to the projection angle (as represented in Figure \ref{triangle}).
\begin{figure}[H]
\centering
\includegraphics[width=0.9\columnwidth]{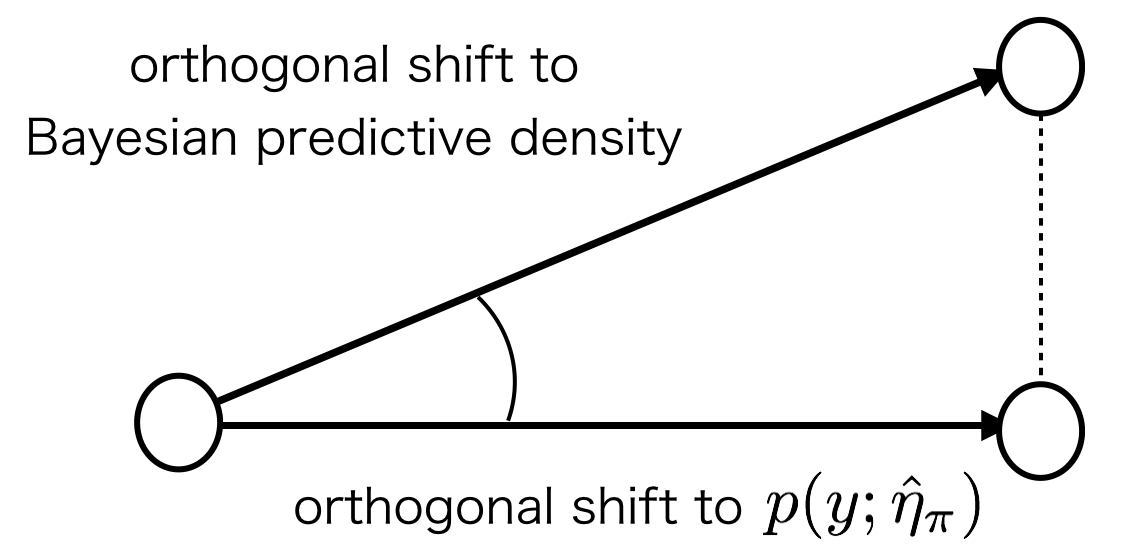}
\caption{
Angle of the projection
}
\label{triangle}
\end{figure}

In the following discussion,
we consider extended plugin densities $p(y;\hat{\eta})$ with estimators $\hat{\eta}=\eta(\hat{\omega},\hat{\xi})$
where $\hat{\omega}$ and $\hat{\xi}$ can be expressed in the following forms, respectively:
\begin{align}
\label{u_and_v}
\hat{\omega} &= \hat{\omega}_{\rm MLE} + \frac{1}{n} \alpha(\hat{\omega}_{\rm MLE}) + \mathrm{o}_p(n^{-1}), \nonumber\\[-3mm]
& \\[-3mm]
\hat{\xi} &=  \frac{1}{n}\beta(\hat{\omega}_{\rm MLE}) + \mathrm{o}_p(n^{-1}). \nonumber
\end{align}
Here, $\alpha^a(\omega), \beta^\kappa (\omega)$ are smooth functions of $\mathrm{O}_p(1)$.
The density can be expanded as
\begin{align}
\label{expansion}
p(y;\hat{\eta}) \nonumber
&=p_{\alpha,\beta}(y;\hat{\omega}_{\rm MLE})\nonumber \\
&=p(y;\hat{\omega}_{\rm MLE}) + \frac{1}{n}\alpha^a(\hat{\omega}_{\rm MLE})\partial_a p(y;\eta(\hat{\omega}_{\rm MLE}))\nonumber\\
& \ \ \ + \frac{1}{n} \beta^\kappa(\hat{\omega}_{\rm MLE})\partial_\kappa p(y;\eta(\hat{\omega}_{\rm MLE})) + \mathrm{o}_p(n^{-1}).
\end{align}
The following results hold also for other asymptotically efficient estimators of $\omega$
other than $\hat{\omega}_{\rm MLE}$, although here we consider only $\hat{\omega}_{\rm MLE}$ for simplicity.
This class of extended plugin densities include $p(x;\hat{\omega}_{\rm MLE})$ and $p(x;\hat{\eta}_\pi)$.
For $\hat{\omega}=\hat{\omega}_{\rm MLE}$ and $\hat{\xi}=0$, the density is the plugin density with the maximum likelihood estimator $\hat{\omega}_{\rm MLE}$.
The extended plugin density $p(y;\hat{\eta}_\pi)$ in Theorem \ref{plug-in expansion} is given by, using {(\ref{eq:hab}) and (\ref{eq:hab1}),}
\begin{align*}
\alpha^a(\hat{\omega}_{\rm MLE}) &= {g^{ab}}(\hat{\omega}_{\rm MLE})\left(\partial_b \log\frac{\pi}{\pi_J} + \frac{T_b}{2} \right),\\
\beta^\kappa &= \frac{1}{2}\mH{ab}{~\kappa}(\hat{\omega}_{\rm MLE})g^{ab}(\hat{\omega}_{\rm MLE}).
\end{align*}

We derive the Kullback--Leibler risk of extended plugin densities in this class.
\begin{prop}
\label{prop_KLrisk}
The Kullback--Leibler risk of an extended plugin density $p(y;\hat{\eta})$ where $\hat{\eta}=(\hat{\omega},\hat{\xi})$ is expressed as (\ref{u_and_v}) is asymptotically expanded as follows:
\begin{align}
\label{eq_KLloss}
&{E}[D\{p(y;\omega), p(y;\hat{\eta})\}]\nonumber \\
&= \frac{d}{2n} + \frac{1}{2n^2}g_{ab}(\omega)\alpha^a(\omega) \alpha^b(\omega) + \frac{1}{n^2}\enabla{a}{}\alpha^a(\omega) \nonumber \\
& \ \ \ +\frac{1}{2n^2}g_{\kappa\lambda}(\omega)\beta^\kappa(\omega) \beta^\lambda(\omega)
 -\frac{1}{2n^2}\mH{ab\kappa}{}(\omega)g^{ab}(\omega)\beta^\kappa(\omega)\notag\\
&~~~+({\rm terms~independent~of~} \alpha, \beta) + \mathrm{o}(n^{-2}),
\end{align}
where 
$\enabla{a}{}\alpha^b = \partial_a \alpha^b + \eGamma{ac}{~b}\alpha^c$.
\end{prop}

\begin{proof}
See Appendix B.
\end{proof}

The risk expansion (\ref{eq_KLloss}) confirms that the risk can be improved from a predictive density in $\mathcal{P}$ by selecting an appropriate orthogonal shift $\beta$.
We obtain the optimal orthogonal shift.
\begin{thm}
\label{optimal orthogonal shift}
The optimal $\beta^\kappa$ in (\ref{expansion}) is given by
\begin{align}
\label{orthogonal_opt}
\beta^\kappa_{\rm opt}(\omega) = \frac{1}{2}\mH{ab}{~\kappa}({\omega})g^{ab}({\omega}).
\end{align}
\end{thm}

\begin{proof}
The risk in Proposition \ref{prop_KLrisk} is
\begin{align*}
&{E}[D\{p(y;\omega), p(y;\hat{\eta})\}]\\
&= \frac{1}{2n^2}g_{ab}\alpha^a\alpha^b + \frac{1}{n^2}\enabla{a}{}\alpha^a(\omega)\\
& \ \ \ + \frac{1}{2n^2}g_{\kappa\lambda}\left(\beta^\lambda -\frac{1}{2}\mH{ab}{~\lambda}g^{ab}\right)\left(\beta^\kappa -\frac{1}{2}\mH{cd}{~\kappa}g^{cd}\right)\\
& \ \ \ 
-\frac{1}{8n^2}\mH{ab}{~\lambda}\mH{cd}{~\kappa}g^{ab}g^{cd}g_{\kappa\lambda}
+ ({\rm terms~independent~of~}\alpha, \beta)\\
& \ \ \ + \mathrm{o}(n^{-2}).
\end{align*}
Therefore, $\beta$ is optimal when
\[
\beta^\kappa(\omega) = \frac{1}{2}\mH{ab}{~\kappa}(\omega)g^{ab}(\omega).
\]
\end{proof}

Accordingly, the orthogonal component of the shift in Theorem \ref{plug-in expansion} is optimal, and the extended plugin density with $\hat{\eta}_\pi$ has the optimal shift (as illustrated in Figure \ref{fig:opt_orth}).
The risk improvement by the optimal shift is evaluated by the inner product of the optimal shifts $\beta^\kappa_{\rm opt}$ and given by
\begin{align*}
\frac{1}{8n^2}\mH{ab}{~\lambda}\mH{cd}{~\kappa}g^{ab}g^{cd}g_{\kappa\lambda}
 + \mathrm{o}(n^{-2}),
\end{align*}
which does not depend on the parallel shift $\alpha$.
Here, $\mH{ab}{~\lambda}\mH{cd}{~\kappa}g^{ab}g^{cd}g_{\kappa\lambda}$ is the mixture mean curvature of $\mathcal{P}$ embedded in $\mathcal{E}$ at $\omega$, thus this risk improvement has a geometrical interpretation.
\begin{figure}[H]
\centering
\includegraphics[width=0.9\columnwidth]{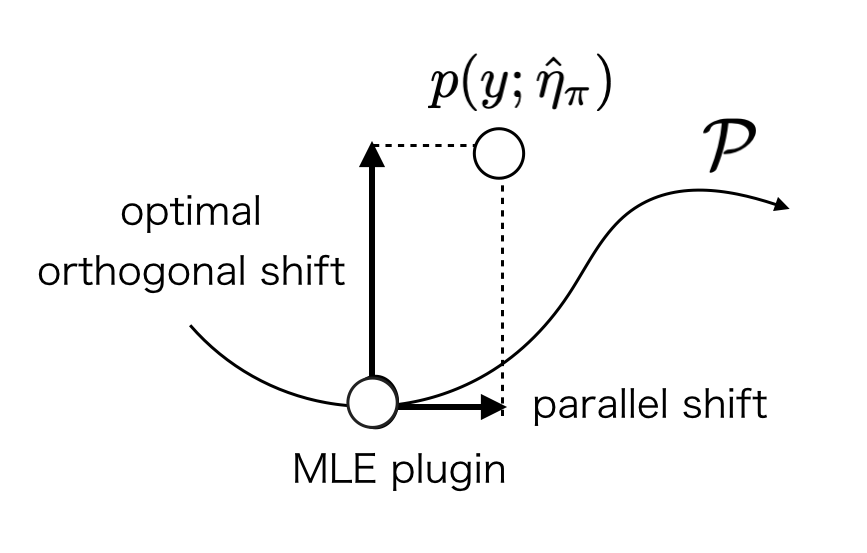}
\caption{
Extended plugin density with $\hat{\eta}_\pi$ has the optimal orthogonal shift
}
\label{fig:opt_orth}
\end{figure}
These results are related to the properties of Bayesian predictive densities.
In \cite{komaki1996}, it is shown that Bayesian predictive densities are optimal along orthogonal shift from $\mathcal{P}$.
The orthogonal shift is not included in the tangent space of $\mathcal{E}$, as explained in Section \ref{subsec:projection}.
The risk improvement achieved by the orthogonal shift is the mixture mean curvature of $\mathcal{P}$ embedded in $\mathcal{F}$, while the risk improvement of $p(y;\hat{\eta}_\pi)$ is the mixture mean curvature of $\mathcal{P}$ embedded in $\mathcal{E}$.
The risk improvement is evaluated by inner products of the optimal shifts, and as a result the cosine of the angle between the two orthogonal shifts, as shown in Figure \ref{triangle}, can be found as the square root of the ratio of the two risk improvements.
\vspace{2mm}

\noindent {Example ~(Fisher circle model, continued)}
We have
$g_{\omega\omega} = 1$ and $\mGamma{\omega\omega}{~\omega}=0$.
Thus, the optimal orthogonal shift is
\begin{align*}
h_{\omega\omega}(x;\eta) &= \left( \partial_{\omega\omega}\eta_i - \mGamma{\omega\omega}{~\omega}\partial_\omega \eta_i \right)\partial_i p(x;\eta)\\
&{
= p(x;\eta)(-(x_1-\eta_1)\cos\omega-(x_2-\eta_2)\sin\omega)
}
\end{align*}
and the risk improvement obtained by the optimal orthogonal shift is
\begin{align*}
&\frac{1}{8n^2}\mH{ab}{~\lambda}\mH{cd}{~\kappa}g^{ab}g^{cd}g_{\kappa\lambda} + \mathrm{o}(n^{-2})\\
&{
= \frac{1}{8n^2}E[(-(y_1-\eta_1)\cos\omega-(y_2-\eta_2)\sin\omega)^2] + \mathrm{o}(n^{-2})
}\\
&= \frac{1}{8n^2} + \mathrm{o}(n^{-2}).
\end{align*}
The risk improvement corresponding to the optimal shift (\ref{orthogonal_komaki}) is ${3}/({8n^2}) + o(n^{-2})$.

If the variance of $x_1, x_2$ is $\sigma^2$,
the risk improvements corresponding to the optimal orthogonal shift and to the shift (\ref{orthogonal_komaki}) can be obtained as follows, respectively:
\[
\frac{\sigma^2}{8n^2},~~\frac{\sigma^2+2}{8n^2}.
\]
Therefore, when $\sigma^2$ is large,
the risk improvement obtained by $p(y;\hat{\eta}_\pi)$ becomes relatively significant as well,
and the performance of the Bayes extended estimator is close to that of the Bayesian predictive density.
The cosine of the angle between the two shift vectors is 
\[
\left.\sqrt{\frac{\sigma^2}{8n^2}}\middle/\sqrt{\frac{\sigma^2+2}{8n^2}}=\sqrt{\frac{\sigma^2}{\sigma^2+2}}\right. ,
\]
which approaches $0$ as $\sigma^2$ increases.
In this way, the Bayesian predictive density and $p(y;\hat{\eta}_\pi)$ approach each other.

\section{Numerical studies}
\label{sec:numerical}
The numerical simulations of the Kullback--Leibler risk of $p(y;\hat{\eta}_\pi)$ and Bayesian predictive densities are shown in a curved Gaussian model.
It confirmed the theoretical results so far and also illustrates the practical importance of the projection of Bayesian predictive densities.

The model $\mathcal{P}$ is the spiked covariance model (for more information about related models, see e.g. \cite{johnstone2001}), that is $l$-dimensional Gaussian $\mathrm{N}_l(0,\Sigma)$ with the covariance matrix $\Sigma$ expressed as
\begin{align}
 \Sigma(\lambda, \bvec{u}) = \lambda \bvec{u} \bvec{u}^\top +I_l,
\end{align}
where the vector $\bvec{u}\in\mathbb{R}^l$ satisfies $\bvec{u}^\top \bvec{u}=1$ and $\lambda>0$.
The eigenvalues of the matrix $\Sigma$ are $\lambda+1, 1, \dots, 1$,
and $\bvec{u}$ is the first eigenvector.

The model $\mathcal{P}$ parametrized by $\omega=(\bvec{u},\lambda)$ is embedded in the larger full exponential family $\mathcal{E}=\{\mathrm{N}_l (0,\Sigma)\mid\Sigma\}$, and the expectation parameter $\eta$ comprises the components of $\Sigma$.
The extended plugin distribution $p(y;\hat{\eta}_\pi)$ is $\mathrm{N}_l(0,\hat{\Sigma}_\pi)$, where $\hat{\Sigma}_\pi$ is the posterior mean of $\Sigma$.
The Bayes estimator $(\hat{\lambda}_\pi,\hat{\bvec{u}}_\pi)$ of $\mathcal{P}$ is composed of $\hat{\lambda}_\pi$, the posterior mean of $\lambda$, and $\hat{\bvec{u}}_\pi$, the first eigen vector of 
$\hat{\Sigma}_\pi$.
The plugin distribution of the Bayes estimator is $\mathrm{N}_l(0,\hat{\lambda}_\pi\hat{\bvec{u}}_\pi\hat{\bvec{u}}_\pi^\top +I_l)$.
The settings of $l$ and $n$ are $(l,n)=(5,20)$ and $(80,320)$.
The posterior mean of $\lambda$ and $\Sigma$ is computed by the 1000 MCMC samples for $l=5$ and 2000 MCMC samples for $l=80$ produced by Gibbs sampling with $250$ and $500$ burn-in samples, respectively.
The Bayesian predictive density is computed { by taking the mean} of the plugin densities of those MCMC samples of $(\lambda,\bvec{u})$.
The Kullback--Leibler risk is derived as the mean of 2000 trials.

\begin{figure}[h]
\centering
\begin{subfigure}{0.9\columnwidth}
\includegraphics[width=\columnwidth]{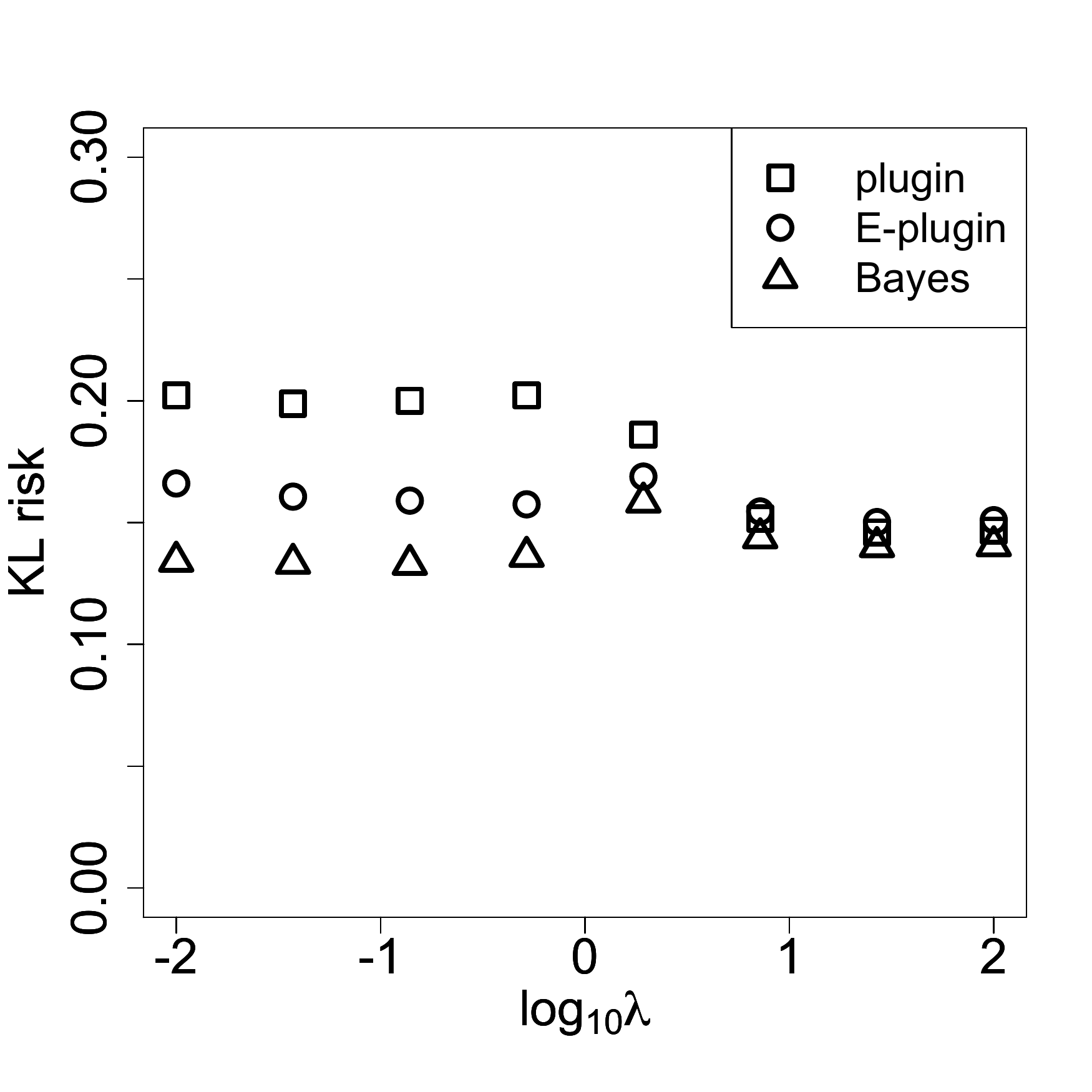}
\caption{l=5, n=20}
\label{p5}
\end{subfigure}
\\
\begin{subfigure}{0.9\columnwidth}
\centering
\includegraphics[width=\columnwidth]{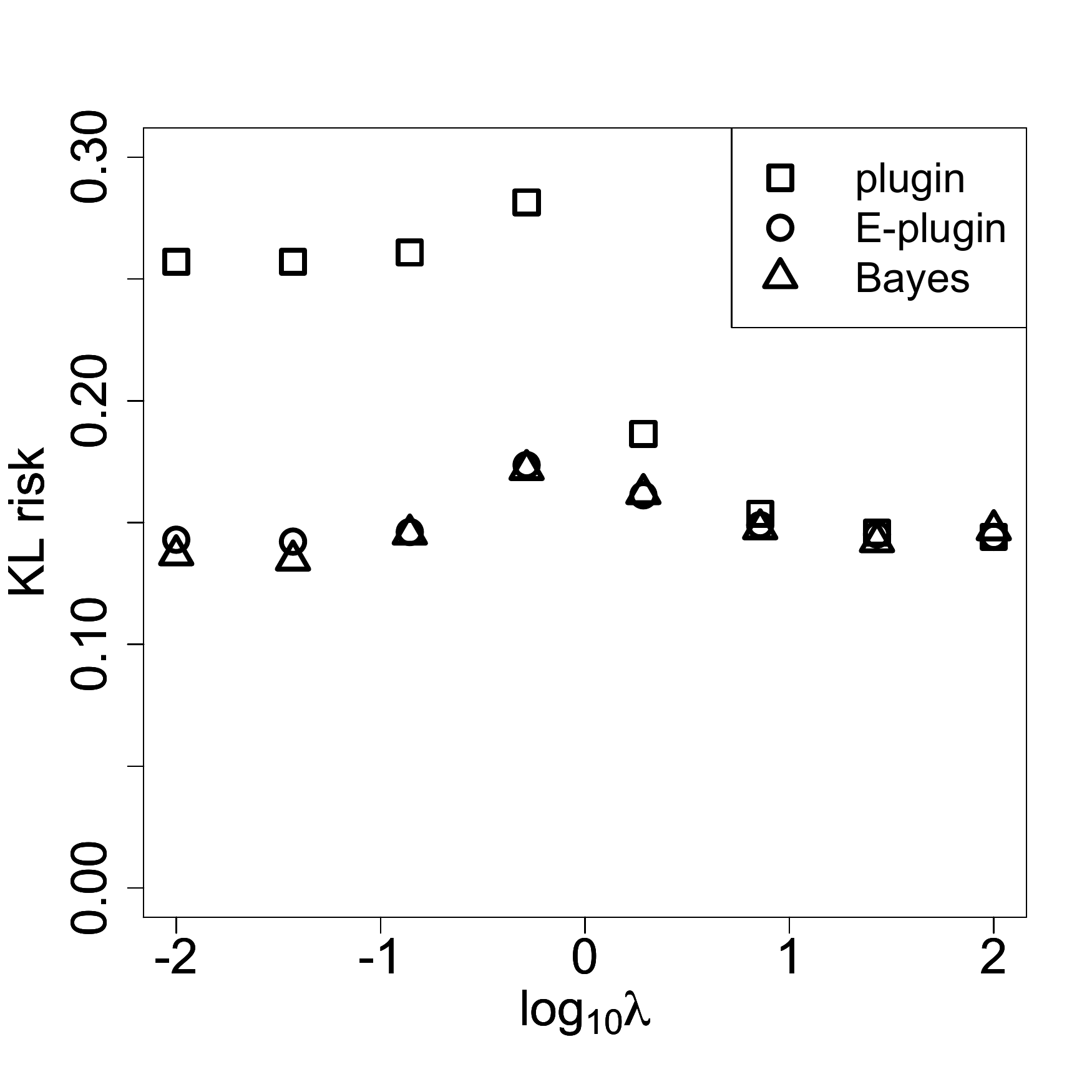}
\caption{l=80, n=320}
\label{p80}
\end{subfigure}
\caption{
Kullback--Leibler (KL) risk of plugin density with Bayes estimator, extended plugin density (E-plugin) with Bayes extended estimator, and Bayesian predictive density
}
\label{fig:KL}
\end{figure}
The results are illustrated in Figure \ref{fig:KL},
and it is confirmed that $p(y;\hat{\eta}_\pi)$ approaches the Bayesian predictive density as the size of $\mathcal{E}$ increases.
In Figure \ref{p5}, the three-layer structure of the plugin density, the extended plugin density, and the Bayesian predictive density is seen in the risk comparison.
On the other hand, in Figure \ref{p80}, it could be seen that the risk plots of $p(y;\hat{\eta}_\pi)$ and the Bayesian predictive density are quite close, which means the projection angle between them is close to zero.
The dimension of the parameter space of $\mathcal{P}$ is $l$ and that of $\mathcal{E}$ is $l(l+1)/2$; thus, $\mathcal{P}$ is embedded in relatively larger full exponential families when $l$ increases.
Therefore, it is natural that the two risk performances approach as $l$ increases because the extended model $\mathcal{E}$ {approaches the set} of all probability densities $\mathcal{F}$, and $p(y;\hat{\eta}_\pi)$ {approaches the Bayesian predictive density}.

\begin{table}[h]
\begin{center}
\caption{{Computational time required to evaluate predictive density for $1000$ new samples and memory size of each predictive density ($l=80, \lambda=1$).}}
  \begin{tabular}{|l|cc|} \hline
 & {Bayes} & {E-plugin} \\ \hline
 {computational time}& {$1.24\times 10^2$ s} & {$1.6\times 10^{-2}$ s}\\ \hline
 {predictive density size}&{$1.29 \times 10^6$ bytes} & {$5.14 \times 10^4$ bytes}\\ \hline
  \end{tabular}
  \label{tab:times}
\end{center}
\end{table}

Figure \ref{p80} also illustrates the practical advantage of $p(y;\hat{\eta}_\pi)$, as it shows that projection of Bayesian predictive densities onto finite-dimensional models of reasonable size is an effective way of approximation.
Bayesian predictive densities are typically approximated by the mean of plugin densities, because obtaining the full density function is intractable.
In these experiments, the Bayesian predictive density is the mean of 2000 plugin densities.
The problem of this approximation is that it requires large space and time to compute density, because storing all MCMC samples is necessary, and we have to {take the mean} of these plugin densities for each $y$.
Figure \ref{p80} demonstrates that the approximation by $p(y;\hat{\eta}_\pi)$, a single point in $\mathcal{E}$ is comparable to the mean of 2000 points in $\mathcal{P}$.
Approximation by $p(y;\hat{\eta}_\pi)$ does not require storing MCMC samples, and the full density function is available avoiding taking mean of plugin densities every time.
{Table \ref{tab:times} provides the computation time required to evaluate the density of 1000 new samples $y$ and the memory size of the Bayesian predictive density and $p(y;\hat{\eta}_\pi)$.
It is shown that both computational time and memory size are effectively saved when we utilize $p(y;\hat{\eta}_\pi)$.}
Therefore, the extended plugin density $p(y;\hat{\eta}_\pi)$ is an effective  approximation of Bayesian predictive densities, with a large advantage in terms of  computational costs, which has a natural theoretical interpretation as a projection onto a finite dimensional model.


%

\appendices
\section{Proof of Theorem \ref{thm_estimator}}
Let $L(\eta)=\frac{1}{n}\sum_{t=1}^{n}\log p(x(t);\eta)$.
Then, 
$\hat{\eta}_\pi(x^n) = \{(\hat{\eta}_\pi)_i(x^n)\}$
is given by
\begin{align*}
\hat{\eta}_\pi(x^n)&=\frac{\int \eta(\omega) p(x^n;\eta(\omega))\pi(\omega)d\omega}{\int p(x^n;\eta(\omega))\pi(\omega)d\omega}\\
&= \frac{\int \eta(\omega) \exp(nL(\eta(\omega)))\pi(\omega)d\omega}{\int \exp(nL(\eta(\omega))) \pi(\omega)d\omega}.
\end{align*}

We approximate $\hat{\eta}_\pi$ using the Laplace method (e.g., \cite{kass1997}, Sec. 3.6 and Sec. 4.6, and \cite{tierney1986}).
The proposed approximation follows Theorem  4.6.1 in  \cite{kass1997}.
First we expand $\pi(\omega)\exp(nL(\omega))$ around $\omega=\hat{\omega}_{\rm MLE}$.
In the following, symbols such as
$\eta(\hat{\omega}_{\rm MLE}), \partial_a \eta(\hat{\omega}_{\rm MLE})$, and $\partial_a\partial_b \eta(\hat{\omega}_{\rm MLE})$
are abbreviated to
$\hat{\eta}, \partial_a\hat{\eta}$, and $\partial_{ab}\hat{\eta}$
respectively.
Let $\omega = \hat{\omega}_{\rm MLE} +\phi/\sqrt{n}$. Then,
\begin{align*}
&\pi(\omega)\exp(nL(\omega))\\
&=
\left(
\hat{\pi} + \frac{(\partial_a\hat{\pi}) \phi^a}{\sqrt{n}} + \frac{(\partial_{ab}\hat{\pi})\phi^a\phi^b}{2n}
+ \frac{(\partial_{abc}\hat{\pi})\phi^a\phi^b\phi^c}{6n\sqrt{n}}\right. \\
&\left. ~~+\mathrm{O}_p(n^{-2})
\right)
\exp
\left(
n\hat{L} +\frac{(\partial_{ab}\hat{L})\phi^a\phi^b}{2}
+ \frac{(\partial_{abc}\hat{L})\phi^a\phi^b\phi^c}{6\sqrt{n}} \right.\\
&\left. ~~+ \frac{(\partial_{abcd}\hat{L})\phi^a\phi^b\phi^c\phi^d}{24{n}} 
+ \frac{(\partial_{abcde}\hat{L})\phi^a\phi^b\phi^c\phi^d\phi^e}{120n\sqrt{n}}\right. \\
&~~ \left. +\mathrm{O}_p(n^{-2})
\right)\\
&=
\hat{\pi}e^{n\hat{L}}e^{(\partial_{ab}\hat{L})\phi^a\phi^b/2} \left(
1 + \frac{(\partial_a \hat{\pi}) \phi^a}{\hat{\pi}\sqrt{n}} + \frac{(\partial_{ab}\hat{\pi})\phi^a\phi^b}{2 \hat{\pi}n}\right. \\
&~~\left. + \frac{(\partial_{abc}\hat{\pi})\phi^a\phi^b\phi^c}{6 \hat{\pi}n\sqrt{n}}+\mathrm{O}_p(n^{-2})
\right)
\left(
1 +  \frac{(\partial_{abc}\hat{L})\phi^a\phi^b\phi^c}{6\sqrt{n}}\right. \\
&~~ \left. + \frac{(\partial_{abcd}\hat{L})\phi^a\phi^b\phi^c\phi^d}{24{n}} \right.\\
&~~ \left. + \frac{(\partial_{abc}\hat{L})(\partial_{a'b'c'}\hat{L})\phi^a\phi^b\phi^c\phi^{a'}\phi^{b'}\phi^{c'}}{72{n}}
+\mathrm{O}_p(n^{-3/2})
\right)\\
\end{align*}
\begin{align*}
&=
\hat{\pi}e^{n\hat{L}}e^{-\hat{J}_{ab}\phi^a\phi^b/2}\left(
1+\frac{(\partial_a \hat{\pi}) \phi^a}{\hat{\pi}\sqrt{n}} 
{+\frac{(\partial_{abc}\hat{L})\phi^a\phi^b\phi^c}{6\sqrt{n}}}\right. \\
&~~ \left. + \frac{(\partial_{ab}\hat{\pi})\phi^a\phi^b}{2\hat{\pi}n}+ \frac{(\partial_a \hat{\pi}) (\partial_{bcd}\hat{L})\phi^a\phi^b\phi^c\phi^d}{6\hat{\pi}n}\right.\\
&~~\left. + \frac{(\partial_{abc}\hat{\pi})\phi^a\phi^b\phi^c}{6\hat{\pi}n\sqrt{n}}
+ \frac{(\partial_{ab}\hat{\pi})(\partial_{cde}\hat{L})\phi^a\phi^b\phi^c\phi^d\phi^e}{12\hat{\pi}n\sqrt{n}}
\right.\\
&~~\left.
+ \frac{(\partial_a \hat{\pi})(\partial_{bcde}\hat{L}) \phi^a\phi^b\phi^c\phi^d\phi^e}{24\hat{\pi}n\sqrt{n}}\right.\\
&~~\left. + \frac{(\partial_a \hat{\pi})(\partial_{bcd}\hat{L})(\partial_{efg}\hat{L}) \phi^a\phi^b\phi^c\phi^d\phi^e\phi^f\phi^g}{72\hat{\pi}n\sqrt{n}}
\right.\\
&~~\left.
{+ \frac{C_1}{n}} + \mathrm{O}_p(n^{-2})
\right).
\end{align*}
Here, $\hat{J}_{ab}=-\partial_{ab}\hat{L}$ and we denote the terms {of $\mathrm{O}_p(n^{-1})$} which do not depend on $\pi$ as
{$C_1/n$, which is $\mathrm{O}_p(n^{-1})$.}
{
Let $(\hat{J}^{ab})$ be the inverse matrix of $(\hat{J}_{ab})$.
Next, we integrate both sides of the above equation with respect to $\omega$.
By changing the variables from $\omega$ to $\phi$, and by using the formula of moments of multivariate Gaussian distributions, we obtain
}
\begin{align*}
&\int \pi(\omega)\exp(nL(\omega)) d\omega\\
&=
C_2\hat{\pi}\left(
1
+ \frac{(\partial_{ab}\hat{\pi})}{2\hat{\pi}n}
\int \phi^a\phi^b e^{-\hat{J}_{cd}\phi^c\phi^d/2} d\phi
\right.\\
&~~\left.
+ \frac{(\partial_a \hat{\pi}) (\partial_{bcd}\hat{L})}{6\hat{\pi}n}
\int \phi^a\phi^b\phi^c\phi^d e^{-\hat{J}_{ef}\phi^e\phi^f/2} d\phi
+ \frac{C_1}{n} \right.\\
&~~\left.
+ \mathrm{O}_p(n^{-2})
\right)
\\
&= C_2\hat{\pi}
\left(
1+\frac{(\partial_{ab}\hat{\pi})\hat{J}^{ab}}{2\hat{\pi}n}
\right.\\
&~~\left.
+ \frac{(\partial_a \hat{\pi}) (\partial_{bcd}\hat{L})(\hat{J}^{ab}\hat{J}^{cd}+\hat{J}^{ac}\hat{J}^{bd}+\hat{J}^{ad}\hat{J}^{bc})}{6\hat{\pi}n} + {\frac{C_1}{n}} 
\right.\\
&~~\left.
+ \mathrm{O}_p(n^{-2}) \right)\\
&= C_2\hat{\pi}
\left(
1+\frac{(\partial_{ab}\hat{\pi})\hat{J}^{ab}}{2\hat{\pi}n}+ \frac{(\partial_a \hat{\pi}) (\partial_{bcd}\hat{L})\hat{J}^{ab}\hat{J}^{cd}}{2\hat{\pi}n} + {\frac{C_1}{n}}
\right.\\
&~~\left.
+ \mathrm{O}_p(n^{-2})
\right).
\end{align*}
{Here $C_2$
 does not depend on $\pi$.
}
Replace $\pi(\omega)$ by $\eta_i(\omega)\pi(\omega)$ and we have
\begin{align*}
&\int \eta_i(\omega)\pi(\omega)\exp(nL(\omega)) d\omega\\
&= C_2\hat{\eta}_i\hat{\pi}
\left(
1+\frac{\{\partial_{ab}(\hat{\eta}_i\hat{\pi})\}\hat{J}^{ab}}{2\hat{\eta}_i\hat{\pi}n}
+ \frac{\{\partial_a(\hat{\eta}_i\hat{\pi})\} (\partial_{bcd}\hat{L})\hat{J}^{ab}\hat{J}^{cd}}{2\hat{\eta}_i\hat{\pi}n}
\right.\\
&~~\left.
+ {\frac{C_1}{n}}
 + \mathrm{O}_p(n^{-2})
\right).
\end{align*}
Therefore, the posterior mean of $\eta_i$ is expanded as
\begin{align*}
&(\hat{\eta}_\pi)_i\\
&= \frac{\int \eta_i \exp(nL(\omega))\pi(\omega)d\omega}{\int \exp(nL(\omega)) \pi(\omega)d\omega}\\
&=
\left.C_2\hat{\eta}_i\hat{\pi}
\left(
1+\frac{\partial_{ab}(\hat{\eta}_i\hat{\pi})\hat{J}^{ab}}{2\hat{\eta}_i\hat{\pi}n}
\right.\right.\\
&~~\left.\left.
+ \frac{\{\partial_a(\hat{\eta}_i\hat{\pi})\} (\partial_{bcd}\hat{L})\hat{J}^{ab}\hat{J}^{cd}}{2\hat{\eta}_i\hat{\pi}n} + {\frac{C_1}{n}}
 + \mathrm{O}_p(n^{-2})
\right) \right.\\
&~~~\left. \middle/ 
C_2\hat{\pi}
\left(
1+\frac{(\partial_{ab}\hat{\pi})\hat{J}^{ab}}{2\hat{\pi}n} + \frac{(\partial_a \hat{\pi}) (\partial_{bcd}\hat{L})\hat{J}^{ab}\hat{J}^{cd}}{2\hat{\pi}n}
\right.\right.\\
&~~\left.\left.
 + {\frac{C_1}{n}}
 + \mathrm{O}_p(n^{-2})
\right) \right.\\
&= \hat{\eta}_i
\left(
1+\frac{\hat{J}^{ab}}{2n}\left(\frac{\partial_{ab}(\hat{\eta}_i\hat{\pi})}
{\hat{\eta}_i\hat{\pi}}-\frac{\partial_{ab}\hat{\pi}}{\hat{\pi}} \right)
\right.\\
&~~\left.
+\frac{\hat{J}^{ab}\hat{J}^{cd}\partial_{bcd}\hat{L}}{2n}
\left(  \frac{\partial_a(\hat{\eta}_i\hat{\pi})}{\hat{\eta}\hat{\pi}}-\frac{\partial_a \hat{\pi}}{\hat{\pi}} \right) +\mathrm{O}_p(n^{-2})
\right)\\
&= \hat{\eta}_i +
\frac{\hat{J}^{ab}}{2n}\left(\partial_{ab}\hat{\eta}_i +
\frac{2 (\partial_a \hat{\eta}_i) (\partial_b \hat{\pi})}{\hat{\pi}} \right)
+\frac{\hat{J}^{ab}\hat{J}^{cd}\partial_{bcd}\hat{L}}{2n} \partial_a \hat{\eta}_i
\\
&~~
+\mathrm{O}_p(n^{-2}).
\end{align*}

Since $\hat{J}_{ab}=\hat{g}_{ab}+o_p(1),~\partial_{bcd}\hat{L}={E}(\partial_{bcd}\hat{L})+o_p(1)$ by the law of large numbers, and
\[
{E}(\partial_{bcd}L) = -\partial_b g_{cd}-\eGamma{cdb}{}=-\partial_b g_{cd}-\mGamma{cdb}{}+T_{bcd},
\]
and
\[
g^{cd}\partial_b g_{cd}=\partial_b \log(|g|)=2\partial_b \log \pi_J
\]
hold, we obtain
\begin{align*}
&(\hat{\eta}_\pi)_i\\
&= \hat{\eta}_i +
\frac{\hat{g}^{ab}}{2n}\left( \partial_{ab}\hat{\eta}_i + 2 \partial_a \hat{\eta}_i \partial_b \log \hat{\pi} \right)
\\
&
~~+\frac{\hat{g}^{ab}}{2n}\left( -2\partial_b \log \hat{\pi}_J - \hat{g}^{cd}\mGamma{cdb}{}(\hat{\omega}_{\rm MLE})+T_{b}(\hat{\omega}_{\rm MLE})
\right) \partial_a \hat{\eta}_i
\\
&
~~+\mathrm{o}_p(n^{-1}) \\
&=\hat{\eta}_i +
\frac{\hat{g}^{ab}}{2n}
\left( \partial_{ab}\hat{\eta}_i - \mGamma{ab}{~c}(\hat{\omega}_{\rm MLE}) \partial_c \hat{\eta}_i \right)
\\
&
~~+\frac{\hat{g}^{ab}}{n}
\left( \partial_b \log \frac{\hat{\pi}}{\hat{\pi}_J} + \frac{T_{b}(\hat{\omega}_{\rm MLE})}{2}
\right)\partial_a \hat{\eta}_i
+\mathrm{o}_p(n^{-1}).
\end{align*}
%

\section{Proof of Proposition \ref{prop_KLrisk}}
We abbreviate symbols such as
$g_{ab}(\omega)$, $\mH{ab\kappa}{}(\omega)$, and $p(y;\omega)$ to
$g_{ab}$, $\mH{ab\kappa}{}$, and $p$, respectively.

The Kullback--Leibler divergence from $p(y;\omega)$ to $p_{\alpha,\beta}(y;\hat{\omega}_\mathrm{MLE})$ is expanded as
{
\begin{align*}
&D(p(y;\omega), p_{\alpha,\beta}(y;\hat{\omega}_\mathrm{MLE}))\\
&=D(p, p(y;\hat{\omega},\hat{\xi}))\\
&=
-\int p\{\log p(y;\hat{\omega},\hat{\xi}) -\log p\} dy,
\end{align*}
and we can expand $\log p(y;\hat{\omega},\hat{\xi}) -\log p$ as follows:
\begin{align*}
&\log p(y;\hat{\omega},\hat{\xi}) -\log p\\
&= \tilde{\omega}^a \partial_a \log p + \hat{\xi}^\kappa \partial_\kappa \log p
+ \frac{1}{2}\tilde{\omega}^a\tilde{\omega}^b \partial_{ab} \log p 
\\
&~~
+ \tilde{\omega}^a\hat{\xi}^\kappa \partial_{a\kappa} \log p
+ \frac{1}{2}\hat{\xi}^\kappa\hat{\xi}^\lambda \partial_{\kappa \lambda}\log p
 +\frac{1}{6}\tilde{\omega}^a\tilde{\omega}^b\tilde{\omega}^c \partial_{abc} \log p
\\
&~~
+\frac{1}{2}\tilde{\omega}^a\tilde{\omega}^b\hat{\xi}^\kappa \partial_{ab\kappa} \log p
+ \frac{1}{24}\tilde{\omega}^a\tilde{\omega}^b\tilde{\omega}^c\tilde{\omega}^d \partial_{abcd} \log p 
\\
&~~
+ \rm{o}_p(n^{-2})
\end{align*}
where $\tilde{\omega} = \hat{\omega} - \omega$.
Because $E[\partial_a \log p]=0, E[-\partial_{ab} \log p]=g_{ab},$ and $g_{a\kappa}=0$ for $a=1,\dots,d,~\kappa=d+1,\dots,m$, we have
\begin{align*}
&D(p(y;\omega), p_{\alpha,\beta}(y;\hat{\omega}_\mathrm{MLE}))\\
&= \frac{1}{2}g_{ab}\tilde{\omega}^a\tilde{\omega}^b + {\frac{1}{2}}g_{\kappa \lambda}\hat{\xi}^\kappa\hat{\xi}^\lambda
- \frac{1}{6}\tilde{\omega}^a\tilde{\omega}^b\tilde{\omega}^c E[\partial_{abc} \log p]
\\
&~~
- \frac{1}{2}\tilde{\omega}^a\tilde{\omega}^b\hat{\xi}^\kappa E[\partial_{ab\kappa} \log p] - \frac{1}{24}\tilde{\omega}^a\tilde{\omega}^b\tilde{\omega}^c\tilde{\omega}^d E[\partial_{abcd} \log p]\\
&~~~ + \rm{o}_p(n^{-2})
\\
&= \frac{1}{2}g_{ab}\tilde{\omega}^a\tilde{\omega}^b + {\frac{1}{2}}g_{\kappa \lambda}\hat{\xi}^\kappa\hat{\xi}^\lambda
 + \left(\frac{1}{2}\mGamma{abc}{}-\frac{1}{3}T_{abc}\right)\tilde{\omega}^a\tilde{\omega}^b\tilde{\omega}^c
 \\
&~~
 + \left\{\frac{1}{2}(\mGamma{ab\kappa}{}+\mGamma{a\kappa b}{}+\mGamma{\kappa ab}{}) - T_{ab\kappa}\right\}\tilde{\omega}^a_{\rm MLE}\tilde{\omega}^b_{\rm MLE}\hat{\xi}^\kappa\\
&~~+ K_{abcd}\tilde{\omega}^a_{\rm MLE}\tilde{\omega}^b_{\rm MLE}\tilde{\omega}^c_{\rm MLE}\tilde{\omega}^d_{\rm MLE} + {\rm o}_p(n^{-2}),
\end{align*}
where 
$\tilde{\omega}_{\rm MLE} = \hat{\omega}_{\rm MLE}-\omega$
and 
}
\begin{align*}
&K_{abcd}\\
&=\frac{1}{24}\int\left\{
6 \Bigl(\frac{\partial_a p}{p}\frac{\partial_b p}{p}\frac{\partial_c p}{p}\frac{\partial_d p}{p} \Bigr)
-12\Bigl(\frac{\partial_a p}{p}\frac{\partial_b p}{p}\frac{\partial_c \partial_d p}{p} \Bigr)
\right. \\
&~~\left.
+ 3\Bigl(\frac{\partial_a \partial_b p}{p}\frac{\partial_c \partial_d p}{p} \Bigr)
+ 4\Bigl(\frac{\partial_a p}{p}\frac{ \partial_b \partial_c \partial_d p}{p} \Bigr)
\right\}p dy.
\end{align*}
Therefore, the Kullback--Leibler risk from $p(y;\omega)$ to $p_{\alpha,\beta}(y;\hat{\omega}_\mathrm{MLE})$ is expanded as
\begin{align*}
&{E}[D\{p(y;\omega), p_{\alpha,\beta}(y;\hat{\omega}_\mathrm{MLE})\}]\\
&= \frac{1}{2}g_{ab}{E}(\tilde{\omega}^a\tilde{\omega}^b) + \frac{1}{2n^2}g_{\kappa \lambda}{E}(\hat{\beta}^\kappa\hat{\beta}^\lambda)
\\
&~~
+ \left(\frac{1}{2}\mGamma{abc}{}-\frac{1}{3}T_{abc}\right){E}
 (\tilde{\omega}^a\tilde{\omega}^b\tilde{\omega}^c)\\
&~~+ \frac{1}{n}\left\{\frac{1}{2}(\mGamma{ab\kappa}{}+\mGamma{a\kappa b}{}+\mGamma{\kappa ab}{}) - T_{ab\kappa}\right\}{E}(\tilde{\omega}_{\rm MLE}^a\tilde{\omega}_{\rm MLE}^b\hat{\beta}^\kappa)\\
&~~+ ({\rm terms~independent~of}~\alpha,\beta) + {\rm o}(n^{-2}).
\end{align*}

Because $\hat{\beta}$ is a smooth function of $\rm{O}_p(1)$ and $\hat{\beta} = \beta + {\rm o}_p(1)$,
\begin{align*}
{E}(\hat{\beta}^\kappa\hat{\beta}^\lambda) &= \beta^\kappa \beta^\lambda + {\rm o}(1),\\
{E}(\tilde{\omega}_{\rm MLE}^a\tilde{\omega}_{\rm MLE}^b\hat{\beta}^\kappa) &= \frac{1}{n}g^{ab}\beta^\kappa + {\rm o}(n^{-1})
\end{align*}
hold and
\begin{align*}
&{E}(\tilde{\omega}^a\tilde{\omega}^b\tilde{\omega}^c)\\
&= \left.({\alpha}^a {E}(\tilde{\omega}_{\rm MLE}^b\tilde{\omega}_{\rm MLE}^c) + {\alpha}^b {E}(\tilde{\omega}_{\rm MLE}^c\tilde{\omega}_{\rm MLE}^a)
\right. \\
& \left. ~~\hspace{0.2mm}
 + {\alpha}^c {E}(\tilde{\omega}_{\rm MLE}^a\tilde{\omega}_{\rm MLE}^b) \right.)/n {+ ({\rm terms~independent~of~}\alpha, \beta)}
\\
& ~~~
+ {\rm o}(n^{-2})\\
&= ({\alpha}^a g^{bc} + {\alpha}^b g^{ca} + {\alpha}^c g^{ab} )/n^2
\\
& ~~~
+ ({\rm terms~independent~of~}\alpha, \beta) + {\rm o}(n^{-2}).
\end{align*}

Lastly we derive ${E}(\tilde{\omega}^a\tilde{\omega}^b)$.
It can be expanded as {(see e.g. (10.19) in \cite{efron1975}, as its multivariate version is presented here)}
\begin{align*}
{E}(\tilde{\omega}^a\tilde{\omega}^b) &= \frac{1}{n}g^{ab} + \frac{1}{n}(g^{ca}\partial_c {E}(\tilde{\omega}^b) + g^{cb}\partial_c {E}(\tilde{\omega}^a))\\
&~~+ {E}\left\{(\tilde{\omega}^a-g^{ac}\partial_c L)(\tilde{\omega}^b-g^{bd}\partial_d L)\right\}.
\end{align*}
From the likelihood equation,
\begin{align*}
&\partial_a L(\hat{\omega}_{\rm MLE}) = 0,\\
&
\partial_a L + \tilde{\omega}_{\rm MLE}^b \partial_{ab} L + \frac{1}{2}\tilde{\omega}_{\rm MLE}^c\tilde{\omega}_{\rm MLE}^d\partial_{acd} L + {\rm o}_p(n^{-1}) = 0,\\
&
g_{ab}\tilde{\omega}_{\rm MLE}^b = \partial_a L + \tilde{\omega}_{\rm MLE}^b(\partial_{ab} L + g_{ab})\\
&\hspace{1.7cm} + \frac{1}{2}\tilde{\omega}_{\rm MLE}^c\tilde{\omega}_{\rm MLE}^d\partial_{acd} L +  {\rm o}_p(n^{-1}).
\end{align*}
{
Because $\tilde{\omega}_{\rm MLE}^b = g^{ab}\partial_a L +{\rm O}_p(n^{-1})$ and $\tilde{\omega}_{\rm MLE} = \hat{\omega}_{\rm MLE}-\omega = \hat{\omega} - \alpha/n -\omega + \rm{o}_p(n^{-1})$,
}
\begin{align*}
g_{ab}\tilde{\omega}_{\rm MLE}^b =& \partial_a L + g^{cd}\partial_d L(\partial_{ac} L + g_{ac})\\
& + \frac{1}{2}\tilde{\omega}_{\rm MLE}^c\tilde{\omega}_{\rm MLE}^d\partial_{acd} L +  {\rm o}_p(n^{-1}),
\\
\tilde{\omega}^b =& \frac{\alpha^b}{n} + g^{ab}\partial_a L
+ g^{ab}g^{cd}\partial_d L(\partial_{ac} L + g_{ac})\\
& + \frac{1}{2}g^{ab}\tilde{\omega}_{\rm MLE}^c\tilde{\omega}_{\rm MLE}^d\partial_{acd} L +  {\rm o}_p(n^{-1}).
\end{align*}
Thus
\[
{E}(\tilde{\omega}^b) = \frac{\alpha^b}{n} + ({\rm terms~independent~of~}\alpha, \beta) + o(n^{-1})
\]
and
\begin{align*}
&{E}\{(\tilde{\omega}^a-g^{ac}\partial_c L)(\tilde{\omega}^b-g^{bd}\partial_d L)\}\\
&=
\frac{\alpha^a\alpha^b}{n^2}
+ \frac{\alpha^a}{n}
\times E \left[ g^{be}g^{cd}\partial_d L(\partial_{ce} L + g_{ce})
\right. \\
&\left.
~~+ \frac{1}{2}g^{be}\tilde{\omega}_{\rm MLE}^c\tilde{\omega}_{\rm MLE}^d\partial_{cde} L\right]
\\
&~~
 \hspace{0.9mm}+\frac{\alpha^b}{n}
E \left[ g^{ae}g^{cd}\partial_d L(\partial_{ce} L + g_{ce})
\right. \\
&\left.
~~+ \frac{1}{2}g^{ae}\tilde{\omega}_{\rm MLE}^c\tilde{\omega}_{\rm MLE}^d\partial_{cde} L\right]\\
&~~
 \hspace{0.9mm}+ ({\rm terms~independent~of~}\alpha, \beta) + {\rm o}(n^{-2}).
\end{align*}
{Using $E[\partial_d L\partial_{ce} L]=\eGamma{ecd}{}/n$ and $E[\partial_{cde} L]= -\eGamma{cde}{}-\eGamma{dec}{}-\eGamma{ecd}{}-T_{ecd}$, we have}
\begin{align*}
&{{E}\{(\tilde{\omega}^a-g^{ac}\partial_c L)(\tilde{\omega}^b-g^{bd}\partial_d L)\}}\\
&=~
\frac{\alpha^a\alpha^b}{n^2} + \frac{\alpha^a}{n^2}E \left[ g^{be}g^{cd}\eGamma{ecd}{} 
\right. \\
&\left.
~~+~\frac{1}{2}g^{be}g^{cd}(-\eGamma{cde}{}-\eGamma{dec}{}-\eGamma{ecd}{}-T_{ecd})\right]
\\
&~~~
+~\frac{\alpha^b}{n^2}E \left[ g^{ae}g^{cd}\eGamma{ecd}{} 
\right. \\
&\left.
~~+~\frac{1}{2}g^{ae}g^{cd}(-\eGamma{cde}{}-\eGamma{dec}{}-\eGamma{ecd}{}-T_{ecd})\right]\\
& ~~~+ ({\rm terms~independent~of~}\alpha, \beta) + {\rm o}(n^{-2})\\
&=~\frac{\alpha^a\alpha^b}{n^2} - \frac{\alpha^a g^{be}g^{cd}\mGamma{cde}{}}{2n^2} - \frac{\alpha^b g^{ae}g^{cd}\mGamma{cde}{}}{2n^2}\\
& ~~~+ ({\rm terms~independent~of~}\alpha, \beta)  + {\rm o}(n^{-2})
\end{align*}
{(in the last equation we use $g^{be}g^{cd}\eGamma{ecd}{} = g^{be}g^{cd}\eGamma{dec}{}$ and $g^{ae}g^{cd}\eGamma{ecd}{} = g^{ae}g^{cd}\eGamma{dec}{}$.)} Hence we obtain
\begin{align*}
&{E}(\tilde{\omega}^a\tilde{\omega}^b)\\
=& \frac{1}{n}g^{ab} + \frac{1}{n^2}\left(g^{ca}\partial_c \alpha^b + g^{cb}\partial_c \alpha^a\right) + \frac{\alpha^a\alpha^b}{n^2}\\
& - \frac{\alpha^a g^{be}g^{cd}\mGamma{cde}{}}{2n^2}- \frac{\alpha^b g^{ae}g^{cd}\mGamma{cde}{}}{2n^2}\\
&+ ({\rm terms~independent~of~}\alpha, \beta)  + {\rm o}(n^{-2}).
\end{align*}

Here,
\[
\mGamma{ab\kappa}{}=\mH{ab\kappa}{}
\]
and
\[
-\mGamma{b\kappa a}{}+T_{ab\kappa}=\mH{ab\kappa}{}
\]
hold.
The first one follows from the definition, and the second one is from the relation
\begin{align*}
&\partial_b g_{a\kappa}\\
&= \int \partial_a\partial_b p \frac{\partial_\kappa p}{p}dy + \int \partial_b\partial_\kappa p \frac{\partial_a p}{p}dy - \int \frac{\partial_a p \partial_\kappa p \partial_b p}{p^2}dy\\
&= \mGamma{ab\kappa}{} + \mGamma{b\kappa a}{} - T_{ab\kappa}
\end{align*}
and $g_{a\kappa}=0$ {($a=1,\dots,d,~\kappa=d+1,\dots,m$).}

Therefore, the risk is expanded as follows:
\begin{align*}
&{E}[D\{p(y;\omega), p_{\alpha,\beta}(y;\hat{\omega}_{\rm MLE})\}]\\
&= \frac{d}{2n} + \frac{g_{ab}}{2n^2}(g^{ca}\partial_c \alpha^b + g^{cb}\partial_c \alpha^a) + \frac{g_{ab}}{2n^2}\alpha^a\alpha^b
\\ &~~~
+\alpha^a\frac{g^{bc}}{2n^2}
\left(\mGamma{abc}{}+\mGamma{bca}{}+\mGamma{cab}{}-2T_{abc} \right)\\
&~~~- \frac{g_{ab}}{2n^2}\alpha^a g^{be}g^{cd}\mGamma{cde}{}+ \frac{g_{\kappa \lambda}}{2n^2}{\beta}^\kappa{\beta}^\lambda
\\ &~~~
+ \beta^\kappa\frac{g^{ab}}{2n^2}\left(\mGamma{\kappa ab}{}+\mGamma{b\kappa a}{}+\mGamma{ab\kappa}{}-2T_{ab\kappa} \right)
\\
 &~~~+ ({\rm terms~independent~of~}\alpha, \beta) + {\rm o}(n^{-2})\\
&= \frac{d}{2n} +\frac{g_{ab}}{2n^2}\alpha^a\alpha^b + \frac{1}{n^2}\left(\alpha^a \eGamma{abc}{}g^{bc}+\partial_a\alpha^a \right)
\\ &~~~
 + \frac{g_{\kappa\lambda}}{2n^2}\beta^\kappa\beta^\lambda-\frac{1}{2n^2}\mH{ab\kappa}{}g^{ab}\beta^\kappa\\
&~~~+ ({\rm terms~independent~of~}\alpha, \beta) + {\rm o}(n^{-2}).
\end{align*}
%

\section*{Acknowledgment}
The authors greatly appreciate the referees' comments on an earlier version.
The authors are grateful to Tomonari Sei for helpful comments.
This paper is based on a part of the first author's  Ph.D. thesis done at the University of Tokyo.

\ifCLASSOPTIONcaptionsoff
  \newpage
\fi



%
\bibliographystyle{IEEEtran}
\bibliography{extended-ref}

%

\begin{IEEEbiographynophoto}{Michiko Okudo}
received the B.E., M.S. and Ph.D. degrees from the
University of Tokyo in 2015, 2017 and 2020, respectively. She is currently
an Assistant Professor with Department of Mathematical Informatics, the
University of Tokyo.
\end{IEEEbiographynophoto}
\enlargethispage{-6.9in}

\begin{IEEEbiographynophoto}{Fumiyasu Komaki}
(M'00) received the B.Eng. and M.Eng. degrees in mathematical engineering both from the University of Tokyo, Japan, in 1987 and 1989, respectively. He received the Ph.D. degree in statistical science in 1992 from the Graduate University for Advanced Studies, Japan. He is currently a Professor with the Department of Mathematical Informatics, the University of Tokyo, and a Unit Leader at RIKEN Center for Brain Science. His interests include prediction theory, Bayesian theory, information geometry, and statistical modeling.
\end{IEEEbiographynophoto}


\end{document}